\documentclass[12pt,a4paper]{article}
\usepackage{amsmath}
\usepackage{amssymb}
\usepackage{cases}
\usepackage{bbm}
\usepackage{mathrsfs}
\usepackage{graphicx}
\usepackage{amsfonts}
\usepackage{enumerate}
\usepackage{theorem}
\usepackage{color}
\usepackage{cite}
\usepackage[pdfstartview=FitH]{hyperref}
\makeatletter
\def\tank#1{\protected@xdef\@thanks{\@thanks
 \protect\footnotetext[0]{#1}}}
\def\bigfoot{

 \@footnotetext}
\makeatother

\newcommand{\ea}{\end{array}}

\allowdisplaybreaks

\newtheorem{theorem}{Theorem}[section]
\newtheorem{lem}{Lemma}[section]
\newtheorem{prp}[theorem]{Proposition}
\newtheorem{thm}[theorem]{Theorem}
\newtheorem{cor}[theorem]{Corollary}
\newtheorem{dfn}[theorem]{Definition}

\newtheorem{remark}{Remark}[section]

\def\beq{\begin{equation}}
\def\nneq{\end{equation}}

\def\bthm{\begin{thm}}
\def\nthm{\end{thm}}

\def\blem{\begin{lem}}
\def\nlem{\end{lem}}
\def\bprf{\begin{proof}}
\def\nprf{\end{proof}}
\def\bprop{\begin{prop}}
\def\nprop{\end{prop}}
\def\brmk{\begin{rem}}
\def\nrmk{\end{rem}}

\def\bexa{\begin{exa}}
\def\nexa{\end{exa}}
\def\bcor{\begin{cor}}
\def\ncor{\end{cor}}

\def\QQ{\mathcal Q}
\def\Rd{\mathbb{R}^d}

\def\K{\mathbf{K}}
\def\B{\mathfrak{B}}
\def\a{\alpha}
\def\b{\beta}
\def\d{\delta}

\def\p{\rho}

\def\e{\varepsilon}

{\theorembodyfont{\rmfamily}
}

\title{Dirichlet boundary value problems for elliptic operators with measure data: a probabilistic approach}

\footnotesize{\author{
 Saisai Yang$^{1,}$\thanks{yangss@mail.ustc.edu.cn},\ \
 Tusheng Zhang$^{2,}$\thanks{Tusheng.Zhang@manchester.ac.uk}\\
 {\em $^1$ School of Mathematical Sciences,}\\
 {\em University of Science and Technology of China,}\\
 {\em Hefei, 230026, China}\\
 {\em $^2$ School of Mathematics, University of Manchester,}\\
 {\em Oxford Road, Manchester, M13 9PL, UK}\\
}
\date{}
\newenvironment{proof}{\par\noindent{\bf Proof:}}{\hspace*{\fill}$\blacksquare$\par}
\begin{document}
\maketitle
%\begin{minipage}{140mm}
%\begin{center}
\noindent \textbf{Abstract:}
In this paper, we use probabilistic approach to prove that there exists a unique weak solution to the Dirichlet
boundary value problem for second order elliptic equations whose coefficients are signed measures, and we will give a probabilistic representation of the solutions.
This is the first time to study Dirichlet boundary value problems with this generality. The heat
kernel estimates play a crucial role in our approach.
%\end{center}

%\end{minipage}

\vspace{4mm}

%\noindent \textbf{AMS Subject Classification}: Primary 60H15 Secondary 35R60, 37L55.

\vspace{3mm}
\noindent \textbf{Key Words:}
Dirichlet
boundary value problem;
signed measures;
probabilistic representation;
heat kernel estimates.
\numberwithin{equation}{section}
\vskip 0.3cm
\noindent \textbf{AMS Mathematics Subject Classification:} Primary 60H60; Secondary 60H30, 31J25.
\section{Introduction}
\indent

Suppose $D$ is a bounded domain in $\mathbb{R}^d$. Let $\mu_i,\nu,\rho$ be signed measures on $\Rd$ for each $1\le i \le d$. In this paper, we are concerned with the following boundary value problem:
\begin{align}\label{1.3}
\left\{
\begin{aligned}
& \mathcal{A} u =-\rho,&&\forall x \in D, \\
& u(x)|_{\partial D}=\varphi(x), && \forall x \in \partial D, \\
\end{aligned}
\right.
\end{align}
where
\begin{align}\label{1.8}
\mathcal{A} u:=\frac{1}{2}\Delta u + \nabla u  \cdot  \mu+  u  \nu.
\end{align}

Let $b=(b_i)_{1 \le i \le d}: D\rightarrow \mathbb{R}^d$ and $q, f: D\rightarrow \mathbb{R}$ are measurable functions in $D$ such that $|b|^2$, $q$ and $f$ belong to the Kato class $\K_{d,2}$ (See the Definition \ref{e1.1} for the precise definition of $\K_{d,2}$) and $\varphi$ is a continuous function on $\partial D$.
%When the general signed measures in (\ref{1.8}) are replaced by functions $b(x)$, $q(x)$ and $f(x)$,
W.D.Gerhard in \cite{Gerhard} proved the following result that there exists a unique continuous weak solution to the Dirichlet boundary value problem:
\begin{align}
\left\{
\begin{aligned}
& \frac{1}{2}\Delta u(x)+ \nabla u(x) \cdot  b(x)+  u(x)q(x) =-f(x),&&\forall x \in D, \\
& u(x)|_{\partial D}=\varphi(x), && \forall x \in \partial D,  \nonumber
\end{aligned}
\right.
\end{align}
under the condition that there exists a $x_0 \in D$ such that
\begin{eqnarray*}\label{1.4}
\begin{split}
E_{x_0}[e^{\int_0^{\tau_D}q(X_s)ds}]<\infty,
\end{split}
\end{eqnarray*}
where $X=(\Omega,\mathcal{F}_t,X_t,P_x,x\in \Rd)$ is the diffusion process associated with the generator $L u :=\frac{1}{2}\Delta u+ \nabla u \cdot  b$, $\tau_D$ is the first exit time of $X$ from $D$ and $E_{x_0}$ is the expectation under $P_{x_0}$. Precisely speaking, there exists a unique $u \in W^{1,2}(D)\bigcap C(\overline{D})$ such that $u=\varphi$ on $\partial D$ and
\begin{eqnarray} \label{1.7} \nonumber
\begin{split}
&\ \ \ \ \frac{1}{2}\sum_{i=1}^{d}\int_{\Rd}\frac{\partial u(x)}{\partial x_i}\frac{\partial \phi(x)}{\partial x_i}dx-\sum_{i=1}^d\int_{\Rd}\frac{\partial u(x)}{\partial x_i}\phi(x)b_i(x)dx -\int_Du(x)\phi (x)q(x)dx \\
&=\int_D  \phi(x)f(x)dx,  \ \ \ \ \ \ \forall \ \phi \in C_0^{\infty}(D).\\
\end{split}
\end{eqnarray}

\vskip 0.3cm

In this paper, we will show that there exists a unique solution to the boundary value problem (\ref{1.3}) under the conditions that for each $1 \le i \le d$, the signed measures $\mu_i$, $\nu$ and $\p$ belong to $\K_{d,\alpha}$ for some $\a<1$, and we will obtain a global $C^{1,\a}$-estimates in this case.

\vskip 0.3cm

Using probabilistic approaches to solve boundary value problems has a long history. The pioneering work goes back to Kukutani \cite{Kakutani} who used Brownian motion to represent the solution of the classical Laplace's equations. In Chen and Zhang \cite{ChenZhang}, they used the Girsanov transformation and the Dirichlet forms methods to obtain the solution to the following problem:
\begin{align} \label{1.10}
\left\{
\begin{aligned}
& \frac{1}{2}\sum_{i,j=1}^{d}\frac{\partial}{\partial x_i}(a_{ij}(x)\frac{\partial u}{\partial x_j})+\sum_{i=1}^db_i(x)\frac{\partial u}{\partial x_i}+q(x)u(x) = 0,&&\forall x \in D, \\
& u(x)|_{\partial D}=\varphi(x), && \forall x \in \partial D,
\end{aligned}
\right.
\end{align}
under some conditions, including that $|b|^2$, $q$ belong to the Kato class $\K_{d,2}$.

\vskip 0.3cm

Note that for any $f\in \K_{d,2}$ and $\epsilon>0$, there exists a $\delta>0$ such that
\begin{eqnarray}\label{1.5}
\int_{\Rd}\phi(x)^2f(x)dx \le \epsilon \int_{\Rd} |\nabla\phi(x)|^2dx + \delta \int_{\Rd} |\phi(x)|^2dx, \ \ \forall \phi \in C_0^{\infty}(\Rd).
\end{eqnarray}
The property (\ref{1.5}) plays an important role in \cite{ChenZhang} and \cite{ChenZhao}. As far as we know, all the previous works about the weak solutions to elliptic PDEs (\ref{1.10}) were considered under the condition that $|b|^2$ satisfies the property (\ref{1.5}). However, there exist functions which are belong to $\K_{d,\a}$ such that the square of them are not locally integrable with regard to ($w.r.t.$) the Lebesgue measure $dx$, hence they fail to hold the property (\ref{1.5}). For example: let $\frac{1}{2}< \a < 1$, $1-\a<\gamma <\frac{1}{2}$ and $g:\mathbb{R}^{d-1} \to \mathbb{R}$, then $f(x_1,x_2,...,x_d):=(x_d-g(x_1,x_2,...,x_{d-1}))^{\gamma-1} \in \K_{d,\a}$ following form the arguments in the proof of Proposition 2.1 and Example 2.4 in \cite{BassChen}. But $f^2$ is not locally integrable obviously. Hence the results in the previous works do not cover our case.

\vskip 0.3cm

Let $W_t$ be a standard Brownian motion in $\Rd$ $w.r.t.$ $P_x$. Following from Remark \ref{R1.1} (i) in the next section, $|b|^2\in \K_{d,2}$ is equivalent to
\begin{eqnarray} \label{1.9}
\lim_{t \to 0}\sup_{x \in \Rd}E_x[\int_0^t|b(W_s)|^2ds]=0,
\end{eqnarray}
which implies that $M_t:=e^{\int_0^tb(W_s)dW_s-\frac{1}{2}\int_0^t|b(W_s)|^2ds}$ is a martingale. Hence the Girsanov transformation can be applied in this case ($e.g.$ see \cite{ChenZhao} and \cite{Gerhard}). However we have seen that (\ref{1.9}) may fail when $f \in \K_{d,\a}$.
Even for $\pi \in \K_{d,\a}$ is absolutely continuous $w.r.t.$ $dx$, it have enormous difficulties to prove that $|\pi|(dx)$ satisfies the Kato-type inequality, thus it is not clear whether the quadratic form associated with the generator $\mathcal{A}$ defined in (\ref{1.8}) is a closed form, $i.e.$ it is lower bounded, closable and satisfies the sector condition (see \cite{Oshima}). Hence the Girsanov transformation and the Dirichlet form methods can not be used in this paper immediately. %Moreover, under this case the criterion in \cite{Trudinger} about the maximum principle in the elliptic PDEs is not met.

On the other hand, we stress also that the maximum principle plays a key role in showing the uniqueness of the solutions to the elliptic PDEs. However, the Trudinger's criterion about the maximum principle in \cite{Trudinger} is not met even when $\nu =0$.

Hence all the previous known methods in solving the existence and uniqueness about the solutions to the elliptic boundary value problems such as those in \cite{ChenZhang}, \cite{ChenZhao}, \cite{Gerhard} and \cite{LuntLyons} ceased to work.

\vskip 0.3cm
In this paper, we will give the definition of the weak solution (see the definition \ref{D2.1} in the next section) to the problem (\ref{1.3}). And we will show that there exists a unique weak solution to the Dirichlet value problem under some general conditions. When $\nu=0$, we will use the heat kernel estimates in \cite{KimSong} to obtain the existence of the solution to the problem (\ref{1.3}). when $\nu \neq 0$, note that $-\nu$ may not be a positive measure, thus there will not exists a Markov process correspond to the generator $\mathcal{A}$. The proposition of continuous additive functionals (CAFs) associated with $\nu$ and $\p$ (see the definition \ref{D2.2} in the next section) will play an important role to show the existence and uniqueness of the solutions to the problem (\ref{1.3}).

% We will use the h-transform method to tackle the singular term $``\rm{div}(\hat b \cdot)"$. To solve the semilinear boundary value problem (\ref{1.1}), we first produce a candidate to the solution by appealing  to the theory of BSDEs. More precisely, we will  solve a class of BSDEs with very singular coefficients and random terminal time. The BSDEs are driven by the martingale part of a diffusion process. The study of this class of BSDEs is of independent interest. It turns out that the classical $L^2$ setting of BSDEs  is not suitable here. We have to work in the framework of $L^1$ and deal with class $D$ stochastic processes. We refer the readers to the nice article \cite{Briand} for related results.
\vskip 0.3cm
The rest of the paper is organized as follows. In Section 2, we set up the precise framework. And we will first establish some estimates in Section 3. Section 4 is devoted to obtaining the existence and uniqueness of solutions to problem (\ref{1.3}) under the case that $\nu=0$. The boundary value problem (\ref{1.3}) will be completely solved in Section 5.

\section{Framework}

Assume that $d \ge 3$. We first recall the following notion.

\begin{dfn} \label{e1.1}
Let $\pi$ be a Radon signed measure on $\Rd$. For $0<\alpha \le 2$, set
\begin{eqnarray*}
\begin{split}
M^{\alpha}_{\pi}(r) :=\sup_{x \in \Rd}\int_{B_r(x)}|x-y|^{\alpha-d}|\pi|(dy),
\end{split}
\end{eqnarray*}
where $|\pi|$ is the total variation of the signed measure $\pi$ and $B_r(x)$ is the ball in $\Rd$ centered at $x$ with radius $r$. We say that $\pi$ is in the Kato class $\mathbf{K}_{d,\alpha}$, if $\lim_{r \to 0}M^{\alpha}_{\pi}(r)=0$. And we say a function $f\in \K_{d,\alpha}$ if the measure $f(x)dx\in \K_{d,\alpha}$.
\end{dfn}
\begin{remark} \label{R1.1}
(i) Set $N^{\alpha,c}_{\pi}(t):=\sup_{x \in \Rd}\int_0^t\int_{\Rd}s^{-\frac{d+2-\a}{2}}e^{-\frac{c|x-y|^2}{2}}|\pi|(dy)ds$. It was proved in \cite{KimSong} that for any $0<\alpha \le 2$ and $c>0$, $\pi \in \mathbf{K}_{d,\alpha}$ if and only if $\lim_{t \to 0}N^{\alpha,c}_{\pi}(t)=0$.

(ii) From the definition of the Kato class, it is easy to see that if $0 < \a < \beta \le 2$, $\K_{d,\a} \subset \K_{d,\beta}$. And for $\pi \in \K_{d,2}$, $|\pi|(D)<\infty$.
\end{remark}

Let $D$ be a bounded, connected $C^{1,1}$-domain and $\varphi \in C(\partial D)$. Hereafter we fix some $0< \a_0 <1$ and let $\mu_i, \nu, \rho \in \K_{d,1}$ for each $1 \le i \le d$. Without loss of generality, we assume the measures above vanish on $D^c$. We will use $ \cdot $ to denote the inner product of the Euclidean space $\mathbb{R}^d$. $| \cdot |$ and $||\cdot||_{C(\overline{D})}$ denote the norm in $\mathbb{R}^d$ and the supremum norm in $C(\overline{D})$ respectively. Let $R_0$ be the diameter of $D$, $i.e.$ $R_0=\sup_{x,y \in D}|x-y|$. For a signed measure $\pi$ on $\Rd$, denote by $\pi^+$ and $\pi^-$ the positive and negative parts of $\pi$ respectively. Let $0 < \a < 1$, recall that the H\"{o}lder space $C^{1,\a }(\overline{D})$ is defined as the subset of $C^1(\overline{D})$ with the norm $||\cdot||_{C^{1,\a}(\overline{D})}$, where
\begin{eqnarray} \nonumber
\begin{split}
&||f||_{C^{1,\a}(\overline{D})}:=||f||_{C(\overline{D})}+\sup_{1 \le i \le d}|||\partial_{x_i} f||_{C(\overline{D})}+\sup_{1 \le i \le d} \sup_{x,y \in \overline{D},x \neq y} \frac{ |\partial_{x_i} f(x)-\partial_{x_i} f(y)|}{|x-y|^{\a}}< \infty.\\
\end{split}
\end{eqnarray}
%||f||_{1,\overline{D}}:=||f||_{\infty}+\sup_{1 \le i \le d}|||\partial_{x_i} f||_{\infty}
For $f \in C^{1,\a}(\partial D)$, define $||f||_{C^{1,\a}(\partial D)}:=\inf_{\Phi}||\Phi||_{C^{1,\a}(\overline{D})}$, where the infimum is taken over the set of all global extensions $\Phi$ of $f$ to $\overline{D}$.

Let $\mathcal{A} u:=\frac{1}{2}\Delta u + \nabla u  \cdot  \mu+  u  \nu$. For $u,v \in C_0^{\infty}(D)$, define the quadratic form
\begin{eqnarray}\label{2.1} \nonumber
\begin{split}
\mathcal{Q} (u,v):=\frac{1}{2}\sum_{i=1}^{d}\int_{D}\frac{\partial u(x)}{\partial x_i}\frac{\partial v(x)}{\partial x_i}dx-\sum_{i=1}^d\int_{D}\frac{\partial u(x)}{\partial x_i}v(x) \mu_i(dx) -\int_Du(x)v(x) \nu(dx).
\end{split}
\end{eqnarray}

\begin{dfn} \label{D2.1}
A function $u\in C^1(D)\bigcap C(\overline{D})$ is called a weak solution to the boundary value problem
\begin{align}\label{2.2}
\left\{
\begin{aligned}
& \mathcal{A} u =-\p,&&\forall x \in D, \\
& u(x)|_{\partial D}=\varphi(x), && \forall x \in \partial D, \\
\end{aligned}
\right.
\end{align}
if $u$ satisfies
\begin{eqnarray}\label{2.3} \nonumber
\begin{split}
\QQ (u,\phi)=\int_D \phi(x)\p (dx) \ \ \mbox{for any} \ \phi \in C_0^{\infty}(D),
\end{split}
\end{eqnarray}
and $u(x)=\varphi(x)$ when $x \in \partial D$.
\end{dfn}
\begin{remark}
Since for each $1 \le i \le d$, the measures $\mu_i,\nu,\rho$ are not absolutely continuous $w.r.t.$ dx, the classical notion of weak solution in the Sobolev space $W^{1,2}(D)$ is not suitable anymore.
\end{remark}
\vskip0.2cm

Fix a nonnegative smooth radial function $\psi \in C_0^{\infty}(B_1(0))$ with $\int_{\Rd} \psi(x)dx=1$. Let $\psi_n(x)=2^{nd}\psi(2^nx)$ and $G_n(x):=(G_n^1(x),G_n^2(x),...,G_n^d(x))$, where
\begin{eqnarray}\label{2.4} \nonumber
\begin{split}
G_n^i(x):=\int_{\Rd}\psi_n(x-y)\mu_i(dy),\ \forall 1 \le i \le d.
\end{split}
\end{eqnarray}

It was shown in \cite{BassChen} that there exists a unique strong Markov process $\{\Omega,\mathcal{F}_t, X_t,P_x,x\in \Rd\}$ such that
\begin{eqnarray}\label{2.5} \nonumber
\begin{split}
X_t=X_0+W_t+A_t,
\end{split}
\end{eqnarray}
where $W_t$ is a $\mathcal{F}_t$- standard Brownian motion in $\Rd$ and $A_t = \lim_{n \to \infty}\int_0^tG_n(X_s)ds$ uniformly in t over finite intervals, where the convergence is in probability. % And there exists a subsequence $\{n_k\}_{k \ge 1}$ such that $\sup_{k \ge 1} \int_0^t|G_{n_k}(Xs)|ds< \infty\  a.e.$ for ant $t>0$.
By Theorem 3.14 in \cite{KimSong}, we know that $X_t$ has a continuous transition density function $p(t,x,y)$ which admits a two-sided Gaussian type estimate, $i.e.$ there exist positive constants $M_i$, $1 \le i \le 5$, such that for any $(t,x,y)\in (0,\infty)\times \Rd \times \Rd$,
\begin{eqnarray}\label{2.8}
\begin{split}
M_1e^{-M_2t}t^{-\frac{d}{2}}e^{-\frac{M_3|x-y|^2}{2t}} \le p(t,x,y) \le M_4e^{M_5t}t^{-\frac{d}{2}}e^{-\frac{M_6|x-y|^2}{2t}}.
\end{split}
\end{eqnarray}

\begin{dfn} \label{D2.2}
Fix a constant $c>M_5$. For any signed measure $\pi$ in $\Rd$, we say that a CAF $B_t$ of $X$ is associated with $\pi$, if
\begin{eqnarray} \nonumber
\begin{split}
B_t=B_t^+-B_t^-, \ \forall t >0,
\end{split}
\end{eqnarray}
where $B_t^+$ and $B_t^-$ are positive CAFs satisfy that for any $x\in \Rd$,
\begin{eqnarray} \nonumber
\begin{split}
E_x[\int_0^{\infty}e^{-ct}dB_t^+]=\int_0^{\infty}e^{-ct}\int_{\Rd}p(t,x,y)\pi^+(dy)dt,\\
E_x[\int_0^{\infty}e^{-ct}dB_t^-]=\int_0^{\infty}e^{-ct}\int_{\Rd}p(t,x,y)\pi^-(dy)dt.
\end{split}
\end{eqnarray}
\end{dfn}

\begin{remark}
By Proposition \ref{P2.2} in the next section, we know that for any $\pi \in \K_{d,2}$ with compact support, there exists a unique CAF $B_t$ associated with $\pi$. Moreover $B_t$ is independent of the choice of the constant $c$.
\end{remark}

% $X_t$ has a continuous transition density function $p(t,x,y)$ which admits a two-sided Aronson's Gaussian type estimate. We know from \cite{BassChen} that for any $\pi \in \K_{d-1}$, there exists an (unique?) CAF $B_t$ associated with it.

\section{Preliminary results}

As a preparation of solving the boundary value problem (\ref{2.2}), in this section, we prove a number of results on the regularity of the heat kernel and establish the existence and properties of the CAFs associated with Kato class measures.

We firstly give the following result concerns general Feller processes.

\begin{lem}\label{L3.8}
Assume that a Feller process $\{Y_t,Q_x,x \in \Rd\}$ has a continuous transition density function $w.r.t.$ the Lebesgue measure, which admits a
Gaussian type lower estimate. Let $\tau_{\overline{D}}$ be the first exit time of $Y_t$ from $\overline{D}$. Then we have
\begin{eqnarray}\label{3.64}
Q_x(\tau_{\overline{D}}=0)=1, \ \forall x \in \partial D.
\end{eqnarray}
and
\begin{eqnarray}\label{3.67}
Q_x(\tau_{\overline{D}}< \infty)=1, \ \forall x \in \overline{D}.
\end{eqnarray}
\end{lem}
\begin{proof}
First we show (\ref{3.64}). By Theorem IV.2.19 in \cite{Kara} we know that for any $x \in \partial D$, there exists a ball $B_r(x)$ such that $Q_x(\tau_{\overline{B_r(x) \cap D}}=0)>0$. By Blumenthal's zero-or-one law for Feller processes (see Theorem 2.3.6 in \cite{ChungWalsh}), we deduce that $Q_x(\tau_{\overline{B_r(x) \cap D}}=0)=1$. Since regularity is a local condition, we have $Q_x(\tau_{\overline{D}}=0)=1$.

\vskip 0.3cm

Now we proof (\ref{3.67}). There exists constants $c_1,c_2>0$ such that the transition density function $g(t,x,y)$ of $Y$ satisfies
\begin{eqnarray} \nonumber
g(t,x,y) \ge c_1t^{-\frac{d}{2}}e^{-\frac{c_2|x-y|^2}{2t}}, \forall (t,x,y) \in (0,1] \times \Rd \times \Rd.
\end{eqnarray}

Thus, we have
\begin{eqnarray}\label{3.68}
\begin{split}
\sup_{x \in \overline{D}} Q_x[Y_1\in \overline{D}]&= 1-\inf_{x \in \overline{D}}Q_x(Y_1 \in \overline{D}^c)\\
&\le 1-\inf_{x \in \overline{D}}\int_{\overline{D}^c}c_1e^{-\frac{c_2|x-y|^2}{2}}dy\\
&\le 1-\int_{\{z:|z|> R_0\}}c_1e^{-\frac{c_2|z|^2}{2}}dz:=\b<1,\\
\end{split}
\end{eqnarray}
where $R_0$ is the diameter of $D$. Hence
\begin{eqnarray}\label{3.69}
\begin{split}
\sup_{x \in \overline{D}}Q_x(\tau_{\overline{D}}> 1)\le \sup_{x \in \overline{D}} Q_x[Y_1\in \overline{D}]\le\b.\\
\end{split}
\end{eqnarray}

By (\ref{3.68}), (\ref{3.69}) and the Markov property of $Y$, we have
\begin{eqnarray} \nonumber
\begin{split}
\sup_{x \in \overline{D}}Q_x(\tau_{\overline{D}}>2)&=\sup_{x \in \overline{D}} Q_x(Y_1\in \overline{D}; \tau_{\overline{D}}>2)\\
&=\sup_{x \in \overline{D}} E^Q_x[Y_1\in \overline{D};Q_{Y_1} [\tau_{\overline{D}}>1]]\\
&\le \b \sup_{x \in \overline{D}} Q_x[Y_1\in \overline{D}]\le \b^2,\\
\end{split}
\end{eqnarray}
where $E^Q_x$ is the expectation under $P_{x_0}$.

By inductions, we have $\sup_{x \in \overline{D}}Q_x(\tau_{\overline{D}}>n) \le \b^n$. Letting $n \to \infty$, we obtain (\ref{3.67}).
\end{proof}

\vskip 0.4cm

Given a signed measure $\pi$ in $\Rd$, set
\begin{eqnarray} \label{2.11}
\begin{split}
H_n(x):=\int_{\Rd}\psi_n(x-y)\pi(dy).
\end{split}
\end{eqnarray}

\begin{lem} \label{L2.1}
Assume $\pi \in \mathbf{K}_{d,\alpha}$ for some $0< \a \le 2$.
Then for any $n \ge 1,r,c>0$, we have
\begin{eqnarray} \label{2.9}
\begin{split}
M_{H_n}^{\a}(r) \le M_{\pi}^{\a}(r),
\end{split}
\end{eqnarray}
and
\begin{eqnarray}
\begin{split} \label{2.10}
\lim_{t \to 0}\sup_{n \ge 1}N_{H_n}^{\a,c}(t)=0,
\end{split}
\end{eqnarray}
where $M_{H_n}^{\a}(r):=M_{H_n(x)dx}^{\a}(r)$ and $N_{H_n}^{\a,c}(t):=N_{H_n(x)dx}^{\a,c}(t)$.
\end{lem}

\begin{proof}
For any $r>0$, by Fubini's theorem,
\begin{eqnarray} \nonumber
\begin{split}
M_{H_n}^{\a}(r)&= \sup_{x \in \Rd}\int_{B_r(x)}|x-y|^{\a-d}|H_n(y)|dy\\
& \le \sup_{x \in \Rd}\int_{B_r(x)}|x-y|^{\a-d}\int_{\Rd} \psi_n(y-z) |\pi| (dz) dy\\
& = \sup_{x \in \Rd} \int_{\Rd}  \int_{B_r(x)} \psi_n(y-z) |x-y|^{\a-d} dy |\pi| (dz)\\
& = \sup_{x \in \Rd} \int_{\Rd}  \int_{B_r(x-z)} \psi_n(y') |x-z-y'|^{\a-d} dy' |\pi| (dz)\\
& = \sup_{x \in \Rd} \int_{\Rd}  \psi_n(y') \int_{B_r(x-y')}  |x-y'-z|^{\a-d}  |\pi| (dz) dy'\\
& \le M_{\pi}^{\a}(r).
\end{split}
\end{eqnarray}
This together with the proof of Proposition 2.2 in \cite{KimSong}, we deduce that for any $c>0$, $\lim_{t \to 0}\sup_{n \ge 1}N_{H_n}^{\a,c}(t)=0$.
\end{proof}

\vskip 0.3cm

Following the arguments of Lemma 3.1 in \cite{ZhangQ} we can also show:
\begin{lem} \label{L3.3}
For any $0 \le \a <1$ and $T>0$, there exists a $\d(T)>0$ such that for any positive measure $\pi \in \K_{d, 1-\a}$, $c>0$ and $(t,x,y) \in (0,T] \times \Rd \times \Rd$,
\begin{eqnarray} \nonumber
\begin{split}
\int_0^t\int_{\Rd}s^{-\frac{d+1+\a}{2}}e^{-\frac{c|x-z|^2}{4s}}(t-s)^{-\frac{d+1}{2}}e^{-\frac{c|z-y|^2}{2(t-s)}} \pi(dz)ds \le \d(T) t^{-\frac{d+1+\a}{2}}e^{-\frac{c|x-z|^2}{4t}} N_{\pi}^{1-\a,\frac{c}{2}}(t).
\end{split}
\end{eqnarray}
\end{lem}

\vskip 0.5cm

Let $X_t^n$ be the solution of the following stochastic differential equation:
\begin{eqnarray}\label{3.8}
\begin{split}
X_t^n=x+W_t+\int_0^tG_n(X_s^n)ds, \ \ \ P_x-a.e..
\end{split}
\end{eqnarray}
Denote by $p^{D}(t,x,y)$ the transition density function of the killed process $X^{D}$ of $X$ upon leaving the domain $D$, $i.e.$ for any Borel set $A \subset D$ and $t>0$, ${P}_x( X^D_t \in A)={P}_x( X_t \in A; t< \tau_D)=\int_Ap^D(t,x,y)dy$. Similar denote by $p^{n,D}(t,x,y)$ the transition density function of the corresponding killed process $X^{n,D}$ of $X^n$.

\vskip 0.3cm

It follows from Theorem 4.2, Theorem 4.6 and Lemma 6.1 in \cite{KimSong} that for $t>0,y\in D$, $p^{D}(t,\cdot,y),p^{n,D}(t,\cdot,y) \in C^{1}(D)$, and moreover, for any $T>0$ there exist constants $M_7(T),M_8(T)>0$, and $M_9,M_{10}>0$, such that for any $n \ge 1$,
\begin{eqnarray}
&p^{D}(t,x,y)\wedge p^{n,D}(t,x,y) \le M_7(T)e^{-M_8(T)t},\  \forall (t,x,y)\in [T,\infty)\times D \times D, \label{3.40}\\
&|\nabla_x p^{D}(t,x,y)| \wedge |\nabla_x p^{n,D}(t,x,y)| \le M_9 t^{-\frac{d+1}{2}}e^{-\frac{M_{10}|x-y|^2}{2t}},\  \forall (t,x,y)\in (0,2]\times D \times D,\label{3.66} \\
&p^D(t,x,y) \wedge p^{n,D}(t,x,y) \le M_9 t^{-\frac{d}{2}}e^{-\frac{M_{10}|x-y|^2}{2t}},\ \forall (t,x,y)\in (0,2]\times D \times D.
\end{eqnarray}

\vskip 0.3cm

Thus, for any $(t,x,y)\in [2,\infty)\times D \times D$,
\begin{eqnarray}\label{3.59}
\begin{split}
&\ \ \ \ |\nabla_x p^{D}(t,x,y)| \wedge |\nabla_x p^{n,D}(t,x,y)|\\
&=|\int_D \nabla_x p^{D}(1,x,z)p^{D}(t-1,z,y)dz| \wedge |\int_D \nabla_x p^{n,D}(1,x,z)p^{n,D}(t-1,z,y)dz|\\
&\le \int_DM_9 M_7(1)e^{-M_8(1)(t-1)}dz\\
&=M_9 M_7(1)e^{M_8(1)}m(D)e^{-M_8(1)t}:=M_{11}e^{-M_8(1)t},
\end{split}
\end{eqnarray}
for some $M_{11}>0$.
Together with (\ref{3.66}), we see that there exists a constant $M_{12}>0$ such that for any $n \ge 1$,
\begin{eqnarray}\label{3.38}
\begin{split}
|\nabla_x p^{D}(t,x,y)|\wedge |\nabla_x p^{n,D}(t,x,y)| \le M_{12} t^{-\frac{d+1}{2}}e^{-\frac{M_{10}|x-y|^2}{2t}},\  \forall (t,x,y)\in (0,\infty)\times D \times D.
\end{split}
\end{eqnarray}
Similarly there exists a constant $M_{13}>0$ such that for any $n \ge 1$,
\begin{eqnarray}
p^D(t,x,y) \wedge p^{n,D}(t,x,y) \le M_{13} t^{-\frac{d}{2}}e^{-\frac{M_{10}|x-y|^2}{2t}},\  \forall (t,x,y)\in (0,\infty)\times D \times D. \label{3.39}
\end{eqnarray}

Note that following the same arguments in \cite{KimSong}, if $\mu_i \in K_{d,1-\a_0}$ for each $1 \le i \le d$, we can choose the constants $M_7-M_{13}$ depending on $\mu$ only via the function $\max_{1 \le i \le d}M_{\mu_i}^{1-\a_0}(\cdot)$. It means that if $\tilde \mu_i \in K_{d,1-\a_0}$ for each $1 \le i \le d$ and
\begin{eqnarray} \nonumber
\begin{split}
\max_{1 \le i \le d}M_{\tilde \mu_i}^{1-\a_0}(r) \le \max_{1 \le i \le d}M_{\mu_i}^{1-\a_0}(r), \ \forall r>0.
\end{split}
\end{eqnarray}
Then (\ref{3.40})-(\ref{3.39}) are still valid with the same constants $M_7-M_{13}$ if $\mu_i$ is replaced by $\tilde \mu_i$ for each $1 \le i \le d$.

For $(t,x,y)\in (0,\infty)\times \overline{D} \times \overline{D}$, set $p^D(t,x,y):=0$ when $x$ or $y$ is on $\partial D$. Then $p^D(t,x,y)$ is continuous on $(0,\infty)\times \overline{D} \times \overline{D}$ by Theorem 4.6 in \cite{KimSong}.

\vskip 0.3cm

\begin{prp} \label{p3.4}
Assume for each $1 \le i \le d$, $\mu_i\in \K_{d,1-\a_0}$. Then for $t>0,y\in D$, $p^{D}(t,\cdot,y) \in C^{1,\a_0}(\overline{D})$. Moreover there exist constants $M_{14},M_{15}>0$, depending on $\mu$ only via the function $\max_{1 \le i \le d}M_{\mu_i}^{1-\a_0}(\cdot)$, such that for any convex subset $D' \subset D$, $(t,x,y),(t,x',y)\in (0,\infty) \times \overline{D'} \times D$ and $1 \le j \le d$,
\begin{eqnarray}\label{3.51}
\begin{split}
|\partial_{x_j} p^{D}(t,x,y)-\partial_{x_j} p^{D}(t,x',y)| \le M_{14}|x-x'|^{\a_0}t^{-\frac{d+1+\a_0}{2}}(e^{-\frac{M_{15}|x-y|^2}{2t}}+e^{-\frac{M_{15}|x'-y|^2}{2t}}).
\end{split}
\end{eqnarray}
\end{prp}

\begin{proof}
Let $q^{D}(s,x,y)$ be the transition density function of the killed Brownian motion $W^{D}$. It is known from Theorem VI.2.1 in \cite{Garroni} that there exist $c_1,c_2>0$ such that for any $(t,x,y)\in (0,1]\times D \times D$ and $1 \le i,j\le d$,
\begin{eqnarray}
& |\partial_{x_i} q^{D}(t,x,y)| \le c_1t^{-\frac{d+1}{2}} e^{-\frac{c_2|x-y|^2}{2t}}, \label{e3.6}\\
& |\partial_{x_i,x_j} q^{D}(t,x,y)| \le c_1t^{-\frac{d+2}{2}} e^{-\frac{c_2|x-y|^2}{2t}}. \label{3.45}
\end{eqnarray}
% & |\partial_{x_i,x_j} q^{D}(t,x,y)-\partial_{x_i,x_j} q^{D}(t,x',y)| \le M_3|x-x'|^{\a_0}t^{-\frac{d+2+\a_0}{2}}(e^{-\frac{M_4|x-y|^2}{2t}}+e^{-\frac{M_4|x'-y|^2}{2t}}).

For any convex subset $D' \subset D$, we define the function $J_{k}(t,x,y)$ recursively for $k \ge 0$ on $(0,1]\times D' \times D$:
\begin{eqnarray} \nonumber
\begin{split}
& J_0(t,x,y):=q^{D}(t,x,y),\\
& J_{k+1}(t,x,y):=\int_0^t\int_D J_k(s,x,z) \nabla_z q^{D}(t-s,z,y) \cdot G_n(z)dzds.\\
\end{split}
\end{eqnarray}

\vskip 0.3cm

We first show that for any $k \ge 0$, $(t,x,y),(t,x',y) \in (0,1]\times D' \times D$ and $1 \le j\le d$,
\begin{eqnarray}\label{e3.7}
\begin{split}
&\ \ \ \ |\partial_{x_j} J_k(t,x,y)-\partial_{x_j} J_k(t,x',y)|\\
&\le 2^{1-\a_0} d^{\frac{\a_0}{2}} c_1 (c_1 \d(1)\sum_{1 \le i \le d} N_{G_n^i}^{1-\a_0,\frac{c_2(1-\a_0)}{2}}(t)) ^k |x-x'|^{\a_0}\\
&\ \ \ \ \times t^{-\frac{d+1+\a_0}{2}}(e^{-\frac{c_2(1-\a_0)|x-y|^2}{4t}}+e^{-\frac{c_2(1-\a_0)|x'-y|^2}{4t}}),
\end{split}
\end{eqnarray}
where $\d(1)$ is defined in Lemma \ref{L3.3}.

If $k=0$, by (\ref{3.45}), $|\partial_{x_j} q^{D}(t,x,y)-\partial_{x_j} q^{D}(t,x',y)| \le  d^{\frac{1}{2}}c_1t^{-\frac{d+2}{2}}|x-x'|$.
Hence together with (\ref{e3.6}), we have
\begin{eqnarray} \nonumber
\begin{split}
&\ \ \ \ |\partial_{x_j} J_0(t,x,y)-\partial_{x_j} J_0(t,x',y)|\\
% &=|\partial_{x_j} q^{D}(t,x,y)-\partial_{x_j} q^{D}(t,x',y)|\\
&\le (|\partial_{x_j} q^{D}(t,x,y)|+|\partial_{x_j} q^{D}(t,x',y)|)^{1-\a_0}|\partial_{x_j} q^{D}(t,x,y)-\partial_{x_j} q^{D}(t,x',y)|^{\a_0}\\
&\le (c_1 t^{-\frac{d+1}{2}}(e^{-\frac{c_2|x-y|^2}{2t}}+e^{-\frac{c_2|x'-y|^2}{2t}}))^{1-\a_0} (d^{\frac{1}{2}}c_1t^{-\frac{d+2}{2}}|x-x'| )^{\a_0}\\
&\le 2^{1-\a_0}d^{\frac{\a_0}{2}}c_1 |x-x'|^{\a_0}t^{\frac{d+1+\a_0}{2}}(e^{-\frac{c_2(1-\a_0)|x-y|^2}{2t}}+e^{-\frac{c_2(1-\a_0)|x'-y|^2}{2t}}),
\end{split}
\end{eqnarray}
which implies (\ref{e3.7}) for $k=0$.

Suppose (\ref{e3.7}) is valid for $k$. Then by Lemma \ref{L3.3},
\begin{eqnarray} \nonumber
\begin{split}
& \ \ \ \ |\partial_{x_j} J_{k+1}(t,x,y)-\partial_{x_j} J_{k+1}(t,x',y)|\\
& \le 2^{1-\a_0}d^{\frac{\a_0}{2}}c_1 (c_1 \d(1) \sum_{1 \le i \le d} N_{G_n^i}^{1-\a_0,\frac{c_2(1-\a_0)}{2}}(t)) ^k |x-x'|^{\a_0}\\
& \ \ \ \ \cdot \sum_{1 \le i \le d} \int_0^t\int_Ds^{-\frac{d+1+\a_0}{2}}(e^{-\frac{c_2(1-\a_0)|x-z|^2}{4s}}+e^{-\frac{c_2(1-\a_0)|x'-z|^2}{4s}})c_1(t-s)^{-\frac{d+1}{2}} e^{-\frac{c_2(1-\a_0)|z-y|^2}{2(t-s)}}|G_n^i(z)|dzds \\
& \le 2^{1-\a_0}d^{\frac{\a_0}{2}}c_1 (c_1 \d(1) \sum_{1 \le i \le d} N_{G_n^i}^{1-\a_0,\frac{c_2(1-\a_0)}{2}}(t)) ^{k+1} |x-x'|^{\a_0}t^{-\frac{d+1+\a_0}{2}}(e^{-\frac{c_2(1-\a_0)|x-y|^2}{4t}}+e^{-\frac{c_2(1-\a_0)|x'-y|^2}{4t}}), \\
\end{split}
\end{eqnarray}
which is (\ref{e3.7}).

\vskip 0.3cm

By Lemma \ref{L2.1}, there exists a $T_0\in (0,1]$ such that $c_3:= \sup_{n\ge 1} c_1 \d(1) \sum_{1 \le i \le d} N_{G_n^i}^{1-\a_0,\frac{c_2(1-\a_0)}{2}}(T_0)$\\
$<1$. By Proposition 2.2 in \cite{KimSong} and (\ref{2.9}) we know that $T_0$ and $c_3$ are dependent on $\mu$ only via the function $\max_{1 \le i \le d}M_{\mu_i}^{1-\a_0}(\cdot)$. Set $c_4:=2^{1-\a_0}d^{\frac{\a_0}{2}}c_1 (1-c_3)^{-1}$.

\vskip 0.3cm

Now we are going to prove that for any $(t,x,y),(t,x',y)\in (0,T_0]\times \overline{D'} \times D$ and $1 \le j \le d$,
\begin{eqnarray}\label{3.50}
\begin{split}
& \ \ \ \ |\partial_{x_j} p^{D}(t,x,y)-\partial_{x_j} p^{D}(t,x',y)|\\
& \le c_4|x-x'|^{\a_0}t^{-\frac{d+1+\a_0}{2}}(e^{-\frac{c_2(1-\a_0)|x-y|^2}{4t}}+e^{-\frac{c_2(1-\a_0)|x'-y|^2}{4t}}).
\end{split}
\end{eqnarray}

Recall the following expansion proved in \cite{KimSong} (see (4.17))
\begin{eqnarray}\nonumber
\begin{split}
\partial_{x_j} p^{n,D}(t,x,y)=\sum_{k \ge 0}\partial_{x_j} J_{k+1}(t,x,y), \ \forall (t,x,y) \in (0,1]\times D \times D.
\end{split}
\end{eqnarray}
Using (\ref{e3.7}), for any $(t,x,y),(t,x',y) \in (0,T_0]\times D' \times D$, we have
\begin{eqnarray}\label{e3.8}
\begin{split}
|\partial_{x_j} p^{n,D}(t,x,y)-\partial_{x_j} p^{n,D}(t,x',y)| &=|\sum_{k \ge 0}(\partial_{x_j} J_{k+1}(t,x,y)-\partial_{x_j} J_{k+1}(t,x',y))|\\
& \le c_4 |x-x'|^{\a_0}t^{-\frac{d+1+\a_0}{2}}(e^{-\frac{c_2(1-\a_0)|x-y|^2}{4t}}+e^{-\frac{c_2(1-\a_0)|x'-y|^2}{4t}}).
\end{split}
\end{eqnarray}
Note that $c_4$ is independent of $n$, and that $\nabla_x p^{n,D}(t,x,y)$ converges to $\nabla_x p^{D}(t,x,y)$ by Theorem 4.5 in \cite{KimSong}. Thus letting $n \to \infty$ in (\ref{e3.8}), we obtain that (\ref{3.50}) is valid for $(t,x,y),(t,x',y) \in (0,T_0]\times D' \times D$. Hence together with the continuity of $p^D(t,x,y)$ on $(0,\infty)\times \overline{D} \times \overline{D}$, we have $p^{D}(t,\cdot,y) \in C^1(\overline{D})$ for $(t,y)\in (0,T_0]\times D$, which deduce that (\ref{3.50}) is also valid for $(t,x,y),(t,x',y) \in (0,T_0]\times \overline{D'} \times D$.

\vskip 0.3cm

Finally we prove (\ref{3.51}).
Following from $p^{D}(t,\cdot,y) \in C^1(\overline{D})$ for $(t,y)\in (0,T_0]\times D$ and Chapman-Kolmogorov equation, it is easy to see that $p^{D}(t,\cdot,y) \in C^1(\overline{D})$ for $(t,y)\in (0,\infty)\times D$.

Moreover, by (\ref{3.40}) and (\ref{3.50}), for any $(t,x,y),(t,x',y) \in [T_0,\infty)\times \overline{D'} \times D$ and $1 \le j\le d$,
\begin{eqnarray} \nonumber
\begin{split}
&\ \ \ \ |\partial_{x_j} p^{D}(t,x,y)-\partial_{x_j} p^{D}(t,x',y)|\\
&=|\int_D(\partial_{x_j} p^{D}(\frac{T_0}{2},x,z)-\partial_{x_j} p^{D}(\frac{T_0}{2},x',z))p^{D}(t-\frac{T_0}{2},z,y)dz|\\
&\le \int_Dc_4|x-x'|^{\a_0}(\frac{T_0}{2})^{-\frac{d+1+\a_0}{2}}(e^{-\frac{c_2(1-\a_0)|x-z|^2}{4{T_0}}}+e^{-\frac{c_2(1-\a_0)|x'-z|^2}{4{T_0}}}) M_7(\frac{T_0}{2})e^{-M_8(\frac{T_0}{2})(t-\frac{T_0}{2})}dz\\
&\le c_4(\frac{T_0}{2})^{-\frac{d+1+\a_0}{2}} M_7(\frac{T_0}{2})e^{-M_8(\frac{T_0}{2})(t-\frac{T_0}{2})} m(D) |x-x'|^{\a_0}.\\
\end{split}
\end{eqnarray}
Hence together with (\ref{3.50}), we obtain (\ref{3.51}).
\end{proof}

\begin{lem}\label{L3.7}
Assume $\pi \in \K_{d,\a}$ for some $0<\a \le 2$, then for any $x,x_1 \in \Rd$ and $r>0$,
\begin{eqnarray}\label{3.37} \nonumber
\begin{split}
\int_{B_{r}(x_1)}|x-y|^{\a-d}|\pi|(dy)\le 2M_{\pi}^\a(r).
\end{split}
\end{eqnarray}
\end{lem}

\begin{proof}
Without loss of generality, we may assume $x_1=0$.

If $r \le \frac{|x|}{2}$, then for any $y \in B_r(0)$, $|x-y|^{\a-d}\le |y|^{\a-d}$. Hence
\begin{eqnarray} \nonumber
\begin{split}
\int_{B_{r}(0)}|x-y|^{\a-d}|\pi|(dy)\le \int_{B_{r}(0)}|y|^{\a-d}|\pi|(dy) \le M_{\pi}^\a(r).
\end{split}
\end{eqnarray}

If $r > \frac{|x|}{2}$ and $y \in {B_{r}(0)}\bigcap B_r(x)^c$, then $|x-y|^{\a-d}\le |y|^{\a-d}$. Hence
\begin{eqnarray} \nonumber
\begin{split}
\int_{B_{r}(0)}|x-y|^{\a-d}|\pi|(dy)&\le \int_{{B_{r}(0)}\bigcap B_r(x)^c}|y|^{\a-d}|\pi|(dy)+\int_{{B_{r}(0)}\bigcap B_r(x)}|x-y|^{\a-d}|\pi|(dy)\\
&\le \int_{B_{r}(0)}|y|^{\a-d}|\pi|(dy)+\int_{B_r(x)}|x-y|^{\a-d}|\pi|(dy) \le 2M_{\pi}^\a(r).
\end{split}
\end{eqnarray}
\end{proof}

\vskip 0.4cm

Recall that $H_n$ is denfined in (\ref{2.11}) and we have the following proposition:

\begin{prp} \label{P2.2}
Fix a constant $c>M_5$, where $M_5$ is the constant appeared in (\ref{2.8}). Then for any signed measure $\pi \in \K_{d,2}$ with compact support, there exists a unique CAF $B_t=B_t^+-B_t^-$ associated with the measure $\pi$ in the sense that, $B_t^+$ and $B_t^-$ are positive CAFs such that for any $x\in \Rd$,
\begin{eqnarray}
E_x[\int_0^{\infty}e^{-ct}dB_t^+]=\int_0^{\infty}e^{-ct}\int_{\Rd}p(t,x,y)\pi^+(dy)dt,\label{2.6}\\
E_x[\int_0^{\infty}e^{-ct}dB_t^-]=\int_0^{\infty}e^{-ct}\int_{\Rd}p(t,x,y)\pi^-(dy)dt.\label{e3.14}
\end{eqnarray}
Moreover, for any $x \in \Rd$ and constant $T>0$,
\begin{eqnarray}
\lim_{n \to \infty}\sup_{t \le T}|B_t-\int_0^tH_n(X_s)ds|=0, \  \mbox{in} \ P_x, \label{3.48}
\end{eqnarray}
and if $x\in D$, then $B_{\tau_D}$ is integrable $w.r.t.$ $P_x$ and
\begin{eqnarray} \label{3.43}
\begin{split}
E_x[B_{\tau_D}]=\int_0^{\infty}\int_{D}p^D(t,x,y)\pi(dy)dt.
\end{split}
\end{eqnarray}
\end{prp}

\begin{proof} Suppose that $\mbox{supp} \  \pi \subset B_r(0)$ for some $r>0$. % $\pi$ has a support in $B_r(0)$
Set $H_n^+(x):=\int_{\Rd}\psi_n(x-y)\pi^+(dy)$ and $H_n^-(x):=\int_{\Rd}\psi_n(x-y)\pi^-(dy)$.

\vskip 0.3cm

We will first show that
\begin{eqnarray}\label{2.18}
\begin{split}
&\ \ \ \ \lim_{n \to \infty} \sup_{x \in \Rd}|\int_0^{\infty}e^{-ct}\int_{\Rd}p(t,x,y)\pi^+(dy)dt-\int_0^{\infty}e^{-ct}\int_{\Rd}p(t,x,y)H_n^+(y)dydt|=0.
\end{split}
\end{eqnarray}

Note that $\int_0^{\infty} t^{-\frac{d}{2}}e^{-\frac{M_6|x-y|^2}{2t}} dt \le c_1|x-y|^{2-d}$ for some $c_1>0$ and $\pi^+ \in \K_{d,2}$. Using (\ref{2.8}) and Lemma \ref{L3.7}, we have that for any $T>0$,
\begin{eqnarray}\label{3.55} \nonumber
\begin{split}
&\ \ \ \  \sup_{x \in \Rd}\int_T^{\infty}e^{-ct}\int_{\Rd}p(t,x,y)\pi^+(dy)dt\\
& \le \sup_{x \in \Rd} \int_T^{\infty} \int_{\Rd} M_4e^{(M_5-c)T}t^{-\frac{d}{2}}e^{-\frac{M_6|x-y|^2}{2t}}\pi^+(dy)dt\\
& \le  M_4e^{(M_5-c)T} \sup_{x \in \Rd}\int_{\Rd} \int_0^{\infty} t^{-\frac{d}{2}}e^{-\frac{M_6|x-y|^2}{2t}} dt \pi^+(dy) \\
& \le M_4e^{(M_5-c)T}c_1\sup_{x \in \Rd} \int_{B_r(0)} |x-y|^{2-d} \pi^+(dy)\\
& \le 2M_4e^{(M_5-c)T}c_1 M_{\pi^+}^2(r).
\end{split}
\end{eqnarray}
Since $\mbox{supp} \  H_n^+ \subset B_{r+1}(0)$ and $M_{H_n^+}^{2}(r+1) \le M_{\pi^+}^{2}(r+1)$ by (\ref{2.9}), we similarly have
\begin{eqnarray}\label{3.56} \nonumber
\begin{split}
&\ \ \ \ \sup_{n \ge 1} \sup_{x \in \Rd}\int_T^{\infty}e^{-ct}\int_{\Rd}p(t,x,y)H_n^+(y)dydt\le 2M_4e^{(M_5-c)T}c_1 M_{\pi^+}^2(r+1).
\end{split}
\end{eqnarray}
It follows that
\begin{eqnarray}\label{2.14}
\begin{split}
\lim_{T \to \infty} \sup_{x \in \Rd}\int_T^{\infty}e^{-ct}\int_{\Rd}p(t,x,y)\pi^+(dy)dt=0,\\
\lim_{T \to \infty}\sup_{n \ge 1}\sup_{x \in \Rd}\int_T^{\infty}e^{-ct}\int_{\Rd}p(t,x,y)H_n^+(y)dydt=0.
\end{split}
\end{eqnarray}

On the other hand, taking into account Remark \ref{R1.1} (i), (\ref{2.8}) and (\ref{2.10}), we can show that
\begin{eqnarray}\label{2.15}
\begin{split}
\lim_{\d \to 0}\sup_{x \in \Rd}\int_0^{{\d}}e^{-ct}\int_{\Rd}p(t,x,y)\pi^+(dy)dt=0,\\
\lim_{\d \to 0}\sup_{n \ge 1} \sup_{x \in \Rd}\int_0^{{\d}}e^{-ct}\int_{\Rd}p(t,x,y)H_n^+(y)dydt=0.
\end{split}
\end{eqnarray}

Fix $\d<T$. For any $R>1$, % take a domain $D_R \supset B_r(0)$ such that $d(D_R, B_r(0)) \ge 4R$, where $d(\cdot,\cdot)$ denote the Euclidean distance in $\Rd$.
note that if $x \in B_{r+4R}(0)^c$, then $ \mbox{supp} \ \pi^+ \bigcup \mbox{supp} \ H_n^+ \subset \{y:|x-y|\ge 3R\}$.  Then similar to the proof of Lemma 3.5 in \cite{KimSong}, we have
\begin{eqnarray} \nonumber
\begin{split}
&\ \ \ \ \lim_{R \to \infty} \sup_{n \ge 1} \sup_{x \in B_{r+4R}(0)^c} \{\int_{{\d}}^{T}e^{-ct}\int_{\Rd}p(t,x,y)\pi^+(dy)dt+\int_{{\d}}^{T}e^{-ct}\int_{\Rd}p(t,x,y)H_n^+(y)dydt\}\\
&\le \lim_{R \to \infty} \sup_{n \ge 1} \sup_{x \in B_{r+4R}(0)^c} \{\int_{{\d}}^{T}\int_{\{y:|x-y|\ge 3R\}}M_4 t^{-\frac{d}{2}}e^{-\frac{M_6|x-y|^2}{2t}}\pi^+(dy)dt\\
&\ \ \ \ \ +\int_{{\d}}^{T} \int_{\{y:|x-y|\ge 3R\}}M_4 t^{-\frac{d}{2}}e^{-\frac{M_6|x-y|^2}{2t}}H_n^+(y)dydt\}\\
&\le c_2M_4T \lim_{R \to \infty} \sup_{x \in \Rd}  \int_{\{y:|x-y|\ge R\}} e^{-\frac{M_6|x-y|^2}{2T}}\pi^+(dy)=0,\\
\end{split}
\end{eqnarray}
where $c_2>0$ is some constant, and the last equality follows from Proposition 2.2 in \cite{KimSong}.
Thus, for any given $\e>0$, we can choose $R$ sufficiently large so that
\begin{eqnarray}\label{2.16}
\begin{split}
\sup_{n \ge 1} \sup_{x \in B_{r+4R}(0)^c} \{\int_{{\d}}^{T}e^{-ct}\int_{\Rd}p(t,x,y)\pi^+(dy)dt+\int_{{\d}}^{T}e^{-ct}\int_{\Rd}p(t,x,y)H_n^+(y)dydt\}\le {\e}.
\end{split}
\end{eqnarray}

Since $f(x,y):=\int_{{\d}}^{T}e^{-ct}p(t,x,y)dt$ is continuous in $B_{r+4R}(0) \times B_{r+1}(0)$, by Lemma 3.3 in \cite{KimSong} we have
\begin{eqnarray}\label{2.17}
\begin{split}
&\ \ \ \ \lim_{n \to \infty} \sup_{x \in B_{r+4R}(0)}|\int_{{\d}}^{T}e^{-ct}\int_{\Rd}p(t,x,y)\pi^+(dy)dt-\int_{{\d}}^{T}e^{-ct}\int_{\Rd}p(t,x,y)H_n^+(y)dydt|\\
&=\lim_{n \to \infty} \sup_{x \in B_{r+4R}(0)}|\int_{B_{r+1}(0)}f(x,y)\pi^+(dy)-\int_{B_{r+1}(0)}f(x,y)H_n^+(y)dy|=0.
\end{split}
\end{eqnarray}

Since $\e$ is arbitrary, combining (\ref{2.14})-(\ref{2.17}), we obtain
\begin{eqnarray} \label{2.19} \nonumber
\begin{split}
\lim_{n \to \infty} \sup_{x \in \Rd}|\int_0^{\infty}e^{-ct}\int_{\Rd}p(t,x,y)\pi^+(dy)dt-\int_0^{\infty}e^{-ct}\int_{\Rd}p(t,x,y)H_n^+(y)dydt|=0,
\end{split}
\end{eqnarray}
which is the desired claim (\ref{2.18}).

\vskip 0.3cm

Obviously, $\int_0^{\infty}e^{-ct}\int_{\Rd}p(t,x,y)H_n^+(y)dydt=E_x[\int_0^{\infty}e^{-ct}H_n^+(X_t)dt]$ is a bounded continuous $c$-excessive function. Hence by (\ref{2.18}), Theorem IV.2.13 and Theorem IV.3.13 in \cite{Blumenthal}, there exists a unique positive CAF $B_t^+$ such that for any $x \in \Rd$, (\ref{2.6}) holds.
Similarly there exists a unique positive CAF $B^-_t$ satisfying (\ref{e3.14}). Hence there exists a unique CAF $B_t:=B_t^+-B_t^-$ is associated with $\pi$.

\vskip 0.3cm

Next we prove (\ref{3.48}).

Note that (\ref{2.6}) and (\ref{2.18}) imply
\begin{eqnarray}\label{2.22}
\begin{split}
\lim_{n \to \infty}\sup_{x \in \Rd}|E_x[\int_0^{\infty}e^{-ct}dB_t^+]-E_x[\int_0^{\infty}e^{-ct}H_n^+(X_t)dt]|=0.\\
\end{split}
\end{eqnarray}
Similarly, we have
\begin{eqnarray}\label{2.23}
\begin{split}
\lim_{n \to \infty}\sup_{x \in \Rd}|E_x[\int_0^{\infty}e^{-ct}dB_t^-]-E_x[\int_0^{\infty}e^{-ct}H_n^-(X_t)dt]|=0.\\
\end{split}
\end{eqnarray}
Combining (\ref{2.22}), (\ref{2.23}) with Lemma 3.10 in \cite{BassChen} we deduce that
\begin{eqnarray} \label{3.47} \nonumber
\begin{split}
\lim_{n \to \infty}\sup_{x \in \Rd}E_x[\sup_{t>0}|\int_0^{t}e^{-cs}dB_s-\int_0^{t}e^{-cs}H_n(X_s)ds|^2]=0.\\
\end{split}
\end{eqnarray}
Consequently (\ref{3.48}) holds.

\vskip 0.3cm

Finally we will show (\ref{3.43}).

Note that by Fubini's theorem,
\begin{eqnarray} \nonumber
\begin{split}
\sup_{n \ge 1}\int_D|H_n(y)|dy\le\sup_{n \ge 1}\int_D\int_{\Rd}\psi_n(y-x)|\pi|(dx)dy=|\pi|(B_r(0)).
\end{split}
\end{eqnarray}
Hence by (\ref{3.40}), for any $T \ge 1$, we have
\begin{eqnarray} \label{3.58} \nonumber
\begin{split}
&\ \ \ \ \sup_{n \ge 1} \sup_{x \in \Rd}|\int_T^{\infty}\int_{D}p^D(t,x,y)\pi(dy)dt+\int_T^{\infty}\int_{D}p^D(t,x,y)H_n(y)dydt|\\
&\le M_7(1)\int_T^{\infty}e^{-M_8(1)t}dt(\sup_{n \ge 1} \int_D|H_n(y)|dy+|\pi|(B_r(0)))\\
&\le 2M_7(1)|\pi|(B_r(0))\frac{e^{-M_8(1)T}}{M_8(1)}.
\end{split}
\end{eqnarray}
Using the above bound and the fact that $p^D(t,x,y)$ is continuous on $(0,\infty) \times \overline{D} \times \overline{D}$, and following the same arguments as in the proof of (\ref{2.18}), we obtain
\begin{eqnarray} \label{3.42}
\begin{split}
\lim_{n \to \infty} \sup_{x \in D} |\int_{0}^{\infty}\int_{D}p^D(t,x,y)\pi(dy)dt-\int_{0}^{\infty}\int_{D}p^D(t,x,y)H_n(y)dydt|=0.
\end{split}
\end{eqnarray}

Set $\bar{H}_n(x):=H_n^+(x)+H_n^-(x)$. By (\ref{2.9}), (\ref{3.39}) and Lemma \ref{L3.7}, we have
\begin{eqnarray} \label{e3.11}
\begin{split}
&\ \ \ \ \sup_{n \ge 1} \sup_{x \in D}\int_0^{\infty}\int_{D}p^D(t,x,y)\bar{H}_n(y)dy dt\\
&\le M_{13}\sup_{n \ge 1} \sup_{x \in \Rd} \int_{B_{r+1}(0)}\int_0^{\infty} t^{-\frac{d}{2}}e^{-\frac{M_{10}|x-y|^2}{2}} dt\bar{H}_n(y)dy\\
&\le M_{13}\sup_{n \ge 1} \sup_{x \in \Rd} \int_{B_{r+1}(0)} c_3|x-y|^{2-d}\bar{H}_n(y)dy\\
&\le 2M_{13}c_3M_{\pi}^2(r+1)< \infty,\\
\end{split}
\end{eqnarray}
for some $c_3>0$.

By (\ref{3.42}), (\ref{e3.11}) and the strong Markov property of $X$, we see that
\begin{eqnarray} \label{e3.12} \nonumber
\begin{split}
&\ \ \ \ \sup_{n \ge 1} \sup_{\{\tau: \ stoping\ time\}}E_x[\int_{\tau}^{\infty}\bar{H}_n(X_s^D)ds|\mathcal{F}_{\tau}]< \infty,\\
% &= \sup_{n \ge 1} \sup_{stoping\ time\ \tau}E_{X_{\tau}}[\int_{0}^{\infty}\bar{H}_n(X_s)ds]
&\ \ \ \ \lim_{n,m \to \infty} \sup_{\{\tau: \ stoping\ time\}}|E_x[\int_{\tau}^{\infty}H_n(X_s^D)ds-\int_{\tau}^{\infty}H_m(X_s^D)ds|\mathcal{F}_{\tau}]|=0.\\
% &= \sup_{n \ge 1} \sup_{stoping\ time\ \tau}E_{X_{\tau}}[\int_{0}^{\infty}H_n(X_s)ds-\int_{0}^{\infty}H_n(X_s)ds]
\end{split}
\end{eqnarray}
Combining this with Proposition I.6.14 in \cite{Bass}, we deduce that for $x \in D$,
\begin{eqnarray} \label{3.57}
\begin{split}
&\ \ \ \ \lim_{n,m \to \infty}E_x[|\int_0^{\tau_D}H_n(X_s)ds-\int_0^{\tau_D}H_m(X_s)ds|^2]\\
&\le \lim_{n,m \to \infty} E_x[\sup_{t > 0}|\int_0^{t}H_n(X_s^D)ds-\int_0^{t}H_m(X_s^D)ds|^2]=0.
\end{split}
\end{eqnarray}

On the other hand, $P_x(\tau_D<\infty)=1$ follows from Lemma \ref{L3.8}. Thus by (\ref{3.48}),
\begin{eqnarray} \label{3.41} \nonumber
\begin{split}
\lim_{n \to \infty} \int_0^{\tau_D}H_n(X_s)ds=B_{\tau_D},\  \mbox{in }P_x.\\
\end{split}
\end{eqnarray}
In view of (\ref{3.42}) and (\ref{3.57}), we can conclude that $B_{\tau_D}$ is integrable and (\ref{3.43}) holds.
\end{proof}

\begin{remark} \label{R3.1}
From the proof of the above proposition, we see that $A_t^i$ is the CAF associated with $\mu_i$ for each $1 \le i \le d$, where $(A_t^1,A_t^2,...,A_t^d)=A_t$.
\end{remark}

The proof of the following result is straightforward. We omit it.

\begin{prp} \label{P3.2}
Let the signed measure $\pi$ and CAF $B_t$ be given as in Proposition \ref{P2.2}. For any bounded measurable function $f$ in $\Rd$, denote by $\tilde B_t$ the CAF associated with $f(x)\pi(dx)$. Then for $x \in \Rd$, we have $P_x$-$a.e.$
%Let a signed measure $\pi \in \K_{d,2}$ has a compact support and the CAF $B_t$ ($w.r.t.$ $X$) is associated with it. Given a bounded Borel measurable function $f$, which is continuous in a domain $E \subset \Rd$.
\begin{eqnarray}\label{e3.15} \nonumber
\begin{split}
\tilde B_t=\int_0^tf(X_s)dB_s, \ \forall t >0.
\end{split}
\end{eqnarray}
\end{prp}

\vskip 0.4cm

The following result is a generalization of the Kahamiskii's inequality.

\begin{lem} \label{L3.6}
Let $U_t$, $t\ge 0$, be a positive $CAF$ of $X$ and $\tau$ a constant or a hitting time. Assume $\b:=\sup_{x \in \Rd}E_x[U_{\tau}]< \infty$. Then for any $n \ge 1$,
\begin{eqnarray} \label{e3.22}
&\sup_{x \in \Rd}E_x[U_{\tau}^n] \le n!\b^n.
\end{eqnarray}
Moreover if $\b<1$, then
\begin{eqnarray}\label{3.65}
\sup_{x \in \Rd}E_x[e^{U_{\tau}}]\le \frac{1}{1-\b}.
\end{eqnarray}
\end{lem}

\begin{proof}
(\ref{3.65}) immediately follows from (\ref{e3.22}). We will prove (\ref{e3.22}) by inductions.

Obviously, (\ref{e3.22}) holds for $n=1$.

Assume now (\ref{e3.22}) holds for $n=k$. Let $\bar U_t$ denote the optional projection of the process $(U_{\tau} \circ \theta_t)^k$. Then by the definition of the optional projection and the strong Markov property of $X$, for any bounded stoping time $\sigma$,
\begin{eqnarray} \label{e3.20}
\bar U_\sigma=E_x[(U_{\tau} \circ \theta_\sigma)^k|\mathcal{F}_\sigma]=E_{X_\sigma}[U_{\tau}^k]\le k!\b^k, \ P_x-a.e..
\end{eqnarray}
Since $\bar U_t$ is an optional process, it follows from (\ref{e3.20}) and Theorem 4.10 in \cite{HeYan} that $P_x$-$a.e.$
\begin{eqnarray} \label{e3.21}
\bar U_t\le k!\b^k, \ \forall t \ge 0.
\end{eqnarray}

Noting that $\tau \le t+\tau\circ \theta_t$ on $\{t < \tau\}$, we have
\begin{eqnarray} \nonumber
U_{\tau}-U_{t}\le U_{t+\tau\circ \theta_t}-U_t=U_{\tau}\circ \theta_t.
\end{eqnarray}
By virtue of (\ref{e3.21}) and Theorem 5.13 in \cite{HeYan}, we have
\begin{eqnarray}\label{e3.23} \nonumber
\begin{split}
&\ \ \ \ E_x[U_{\tau}^{k+1}]\\
&= (k+1)E_x[\int_0^{\tau}(U_{\tau}-U_{t})^kdU_{t}]\\
&\le (k+1)E_x[\int_0^{\tau}(U_{\tau}\circ \theta_t)^kdU_{t}]\\
&= (k+1)E_x[\int_0^{\tau} \bar U_t dU_{t}]\\
&\le (k+1)!\b^k E_x[U_{\tau}]\le (k+1)!\b^{k+1}.
\end{split}
\end{eqnarray}
We have proved (\ref{e3.22}) by induction.
\end{proof}

\vskip 0.4cm

\begin{prp} \label{P3.5}
Let the signed measure $\pi$ and CAF $B_t$ be given as in Proposition \ref{P2.2}. $|B|_t$ denotes the total variation of $B_t$. Then for $T,c>0$ and $n \ge 1$,
\begin{eqnarray}
&E_{x}[|B|_T]=\int_0^{T} \int_{\Rd}p(t,x,y)|\pi|(dy)dt, \label{3.44}\\
&\sup_{x \in \Rd}E_x[|B|_T^n] <\infty,  \label{3.61} \\
&\sup_{x \in \Rd}E_x[e^{c|B|_T}] <\infty,\label{3.62}\\
&\lim_{t \to 0}\sup_{x \in \Rd}E_x[e^{c|B|_t}]=1.\label{3.63}
\end{eqnarray}
\end{prp}

\begin{proof}
Let $E:=\mbox{supp} \ \pi^+$. $B^+_t$ and $B^-_t$ denote the CAFs associated with $\pi^+$ and $\pi^-$ respectively. By Proposition \ref{P3.2},
\begin{eqnarray} \nonumber
\begin{split}
E_x[\int_0^{\infty}e^{-ct}I_E(X_t)dB^-_t]=\int_0^{\infty}e^{-ct}\int_{\Rd}p(t,x,y)I_E(y)\pi^-(dy)dt=0.
\end{split}
\end{eqnarray}
Hence $P_x$-$a.e.$,
\begin{eqnarray} \nonumber
\begin{split}
\int_0^tI_E(X_s)dB^-_s=0,\ \forall t>0,
\end{split}
\end{eqnarray}
which implies $\int_0^tI_E(X_s)dB_s$ is a positive CAF. Combining this with the fact that
\begin{eqnarray} \nonumber
\begin{split}
E_x[\int_0^{\infty}e^{-ct}I_E(X_t)dB_t]&=\int_0^{\infty}e^{-ct}\int_{\Rd}p(t,x,y)I_E(y)\pi^+(dy)dt\\
&=\int_0^{\infty}e^{-ct}\int_{\Rd}p(t,x,y)\pi^+(dy)dt,
\end{split}
\end{eqnarray}
we have $\int_0^tI_E(X_s)dB_s=B^+_t$. Similarly we have $\int_0^tI_{E^c}(X_s)dB_s=-B^-_t$. Hence $|B|_t=B^+_t+B^-_t$, which is the CAF associated with $|\pi|$. Similar to the proof of (\ref{3.43}), we have (\ref{3.44}).

(\ref{3.61})-(\ref{3.63}) follow from Lemma \ref{L3.6} and (\ref{3.44}).
\end{proof}

\vskip 0.4cm

Let $\pi \in \K_{d,1}$. For any $r>0$ and $x_0 \in \Rd$, recall that $q^{B_r(x_0)}(t,x,y)$ is the transition density of the killed Brownian motion $W^{B_r(x_0)}$. Define $R_{B_r(x_0)}\pi(x):=\int_0^{\infty}\int_{B_r(x_0)}q^{B_r(x_0)}(t,x,y)\pi(dy)dt$ and the measure $\mathfrak{B}_{B_r(x_0)}\pi:=(\nabla R_{B_r(x_0)}\pi) \cdot \mu_{B_r(x_0)}$, where $\mu_{B_r(x_0)}(dx):=I_{B_r(x_0)}(x)\mu(dx)$.

\begin{lem} \label{L3.5}
There exist constants $M_{16}>0$ and $r_0<1$, depending on $\mu$ only via the function $\max_{1 \le i \le d}M_{\mu_i}^1(\cdot)$, such that for any $\pi \in \K_{d,1}$, $r\le r_0$ and $x_0 \in \Rd$,
\begin{eqnarray} \nonumber
&\sup_{x \in B_r(x_0)}|\nabla R_{B_r(x_0)}\pi(x)| \le M_{16}M^1_{\pi}(r),\\
&M_{\B_{B_r(x_0)} \pi}^1(r)\le \frac{1}{2}M_{\pi}^1(r). \nonumber
\end{eqnarray}
\end{lem}

\begin{proof}
By Theorem VI.2.1 in [8] and Lemma 6.1 in \cite{KimSong}, there exist constants $c_i>0$, $1\le i \le 3$, such that
\begin{eqnarray}\label{3.53} \nonumber
\begin{split}
|\nabla_x q^{{B_1(0)}}(t,x,y)| \le c_1t^{-\frac{d+1}{2}}e^{-\frac{c_2|x-y|^2}{2t}},\  \forall (t,x,y)\in (0,2]\times {B_1(0)} \times {B_1(0)}, \\
q^{B_1(0)}(t,x,y) \le c_1 e^{-c_3t},\ \forall (t,x,y)\in [1,\infty)\times {B_1(0)} \times {B_1(0)}. \\
\end{split}
\end{eqnarray}
Similar to (\ref{3.38}), we can find a $c_4>0$ such that
\begin{eqnarray}\label{3.60}
\begin{split}
|\nabla_x q^{{B_1(0)}}(t,x,y)| \le c_4t^{-\frac{d+1}{2}}e^{-\frac{c_2|x-y|^2}{2t}},\  \forall (t,x,y)\in (0,\infty)\times {B_1(0)} \times {B_1(0)}.
\end{split}
\end{eqnarray}

For $x_0\in \Rd$ and $r<1$, set $B_r:=B_{r}(x_0)$. By the scaling property of the Brownian motion, it is known that
$q^{B_r}(t,x,y)=r^{-d}q^{B_1(0)}(\frac{t}{r^2},\frac{x-x_0}{r},\frac{y-x_0}{r})$.
Hence (\ref{3.60}) yields
\begin{eqnarray} \nonumber
\begin{split}
|\nabla_x q^{B_r}(t,x,y)| \le c_4t^{-\frac{d+1}{2}}e^{-\frac{c_2|x-y|^2}{2t}},\  \forall (t,x,y)\in (0,\infty)\times B_r \times B_r.\\
\end{split}
\end{eqnarray}

Hence by Lemma \ref{L3.7}, we have
\begin{eqnarray} \nonumber
\begin{split}
\sup_{x \in B_r}|\nabla R_{B_r}\pi(x)|&\le \sup_{x \in B_r} \int_0^{\infty}\int_{B_r} |\nabla_x  q^{B_r}(t,x,y)| |\pi|(dy)dt\\
&\le \sup_{x \in \Rd} \int_{B_r} \int_0^{\infty} c_4t^{-\frac{d+1}{2}}e^{-\frac{c_2|x-y|^2}{2t}} dt|\pi|(dy)\\
&\le \sup_{x \in \Rd} \int_{B_r} c_4c_5 |x-y|^{1-d} |\pi|(dy)\\
&\le 2c_4c_5 M^1_{\pi}(r).
\end{split}
\end{eqnarray}
for some $c_5>0$.

Take $r_0>0$, such that $\sum_{1 \le i \le d}M_{\mu_i}^{1}(r_0)\le \frac{1}{4c_4c_5}$. Then for any $r \le r_0$,
\begin{eqnarray} \nonumber
\begin{split}
M_{\B_{B_r} \pi}^1(r)&\le \sup_{x \in B_r}|\nabla R_{B_r}\pi(x)| \sum_{1 \le i \le d}M_{\mu_i}^{1}(r)\\
&\le 2c_4c_5M^1_{\pi}(r) \sum_{1 \le i \le d}M_{\mu_i}^{1}(r)\\
&\le \frac{1}{2} M^1_{\pi}(r).
\end{split}
\end{eqnarray}
\end{proof}

\vskip 0.4cm

\section{Case for $\nu$=0}

In this section, we will establish the existence and uniqueness of solution to the following boundary problem:
\begin{align}\label{3.1}
\left\{
\begin{aligned}
& \frac{1}{2}\Delta u+ \nabla u \cdot  \mu =-\p,&&\forall x \in D, \\
& u(x)|_{\partial D}=\varphi(x), && \forall x \in \partial D. \\
\end{aligned}
\right.
\end{align}

Recall that $\varphi \in C(\partial D)$. Set ${\tilde{u}}(x):=E_x[\varphi(X_{\tau_D})]$. We have the following result.

\begin{lem}\label{L4.1}
For any $x_0 \in \partial D$, we have
\begin{eqnarray} \nonumber
\lim_{D \ni x \to x_0}\tilde u(x) =\tilde u(x_0)=\varphi(x_0).
\end{eqnarray}
\end{lem}

\begin{proof}
Set ${\tilde{u}}_0(x):=E_x[\varphi(W_{\tau_D})]$. It is well known that ${\tilde{u}}_0$ is continuous on $\overline{D}$ and $\tilde u_0(x_0)=\varphi(x_0)$ for $x_0 \in \partial D$. For any $\d>0$ and $x \in D$, we have
\begin{eqnarray}\label{3.2}
\begin{split}
&\ \ \ \ {\tilde{u}}(x)-{\tilde{u}}_0(x)\\
&=E_x[{\tilde{u}}_0(X_{\tau_D})]-{\tilde{u}}_0(x)\\
&=E_x[{\tilde{u}}_0(X_{\tau_D});\tau_D>\delta]+E_x[{\tilde{u}}_0(X_{\tau_D})-{\tilde{u}}_0(x);\tau_D \le \delta]-{\tilde{u}}_0(x)P_x(\tau_D > \delta)\\
&\le 2||{\tilde{u}}_0||_{C(\overline{D})} P_x(\tau_D > \delta)+E_x[\sup_{s \le \delta}|{\tilde{u}}_0(X_s)-\tilde{u}_0(x)|;\tau_D \le \delta].\\
\end{split}
\end{eqnarray}

Fix $x_0\in \partial D$. Since $(X_t,P_x)$ is strong Feller and $x_0$ is regular for $D^c$, similar to the proof of Proposition 4.4.1 in \cite{ChungWalsh} we can show that $x \mapsto P_{x}(\tau_D> \d)$ is upper semi-continuous in $\Rd$. Hence
\begin{eqnarray}\label{3.3}
\lim_{D\ni x \to x_0}P_x(\tau_D > \delta)=0.
\end{eqnarray}

For any $\e>0$, since $\tilde{u}_0$ is uniformly continuous on $\overline{D}$, there exists a $\b >0$ such that for any $x,y \in \overline{D}$ with $|x-y| < \b$, we have
\begin{eqnarray}\label{4.17}
|\tilde{u}_0(x)-\tilde{u}_0(y)| \le \e.
\end{eqnarray}
By Remark \ref{R3.1}, (\ref{3.44}) and the Doob's inequality,
\begin{eqnarray}\label{4.13}  \nonumber
\begin{split}
&\ \ \ \ \lim_{\d \to 0}\sup_{x \in \Rd}P_x[\sup_{t \le \d}|X_t-x|\ge \b]\\
&\le \lim_{\d \to 0}\sup_{x \in \Rd}P_x[\sup_{t \le \d}|W_t|\ge \frac{\b}{2}]+\lim_{\d \to 0}\sup_{x \in \Rd}P_x[\sup_{t \le \d}|A_t|\ge \frac{\b}{2}]\\
&\le \frac{16}{\b^2} \lim_{\d \to 0}\sup_{x \in \Rd}E_x[|W_\d|^2]+\frac{2}{\b}\lim_{\d \to 0}\sup_{x \in \Rd}E_x[|A|_\d]\\
&\le \frac{16}{\b^2} \lim_{\d \to 0}\sup_{x \in \Rd}E_x[|W_\d|^2]+\frac{2}{\b}\lim_{\d \to 0}\sup_{x \in \Rd}\sum_{1\le i \le d}\int_0^\d\int_{\Rd}p(s,x,y)|\mu_i|(dy)ds=0.\\
\end{split}
\end{eqnarray}
Thus there exists a constant $\d >0$ such that
\begin{eqnarray}\label{3.4} \nonumber
\sup_{x \in \Rd}P_x(\sup_{t \le \d} |X_t-x|\ge \b) \le \e.
\end{eqnarray}
Hence, in view of (\ref{4.17}), for $x \in D$,
\begin{eqnarray}\label{3.5}
\begin{split}
&\ \ \ \ E_x[\sup_{s \le \delta}|\tilde{u}_0(X_s)-\tilde{u}_0(x)|;\tau_D \le \delta]\\
&\le  E_x[\sup_{s \le \delta}|\tilde{u}_0(X_s)-\tilde{u}_0(x)|;\sup_{t \le \delta}|X_t-x|\ge \beta]+E_x[\sup_{s \le \delta}|\tilde{u}_0(X_s)-\tilde{u}_0(x)|;\sup_{t \le \delta}|X_t-x|< \beta]\\
&\le  2 ||\tilde{u}_0||_{C(\overline{D})} P_x[\sup_{t \le \delta}|X_t-x|\ge \beta]+ E_x[\sup_{|x-y| \le \b}|\tilde{u}_0(y)-\tilde{u}_0(x)|;\sup_{t \le \delta}|X_t-x|< \beta]\\
&\le 2\e ||\tilde{u}_0||_{C(\overline{D})}+\e.
\end{split}
\end{eqnarray}

Combining (\ref{3.2}), (\ref{3.3}) with (\ref{3.5}), we arrive at
\begin{eqnarray}\label{3.6}
\lim_{D\ni x \to x_0 }|\tilde{u}(x)-\tilde{u}_0(x)|\le 2\e ||\tilde{u}_0||_{C(\overline{D})}+\e.
\end{eqnarray}
Since $\lim_{D \ni x \to x_0}\tilde{u}_0(x)=\varphi(x_0)=\tilde u (x_0)$ by Lemma \ref{L3.8}, letting $\e \to 0$ in (\ref{3.6}), we see that
$\lim_{D \ni x \to x_0}\tilde{u}(x)=\tilde u(x_0)=\varphi (x_0)$.
\end{proof}

\vskip 0.4cm

Next we will show that $\tilde u$ is a weak solution to the problem:
\begin{align} \nonumber
\left\{
\begin{aligned}
&\frac{1}{2}\Delta \tilde u + \nabla \tilde u  \cdot  \mu  =0,&&\forall x \in D, \\
& \tilde u(x)|_{\partial D}=\varphi(x), && \forall x \in \partial D. \\
\end{aligned}
\right.
\end{align}
To this end, we will approximate $\tilde u$ by approximate smooth functions.

Recall that $X_t^n$ is defined in (\ref{3.8}). Set $\tilde{u}_n(x):=E_x[\varphi(X^n_{\tau_D^n})]$, where $\tau_D^n$ is the first exit time of $X^n_t$. We are going to show that $\tilde u_n \to \tilde u$ and $\nabla \tilde u_n \to \nabla \tilde u$ uniformly on compact subsets of $D$. This will be done in the following two lemmas.

\vskip 0.4cm

\begin{lem}\label{Lemma 3.2}
For any $x \in D$, we have
\begin{eqnarray}\label{e3.24}
\begin{split}
\lim_{n \to \infty}\tilde{u}_n(x)=\tilde{u}(x).
\end{split}
\end{eqnarray}
\end{lem}

\begin{proof}
It is known from \cite{BassChen} and \cite{KimSong} that $X_{\cdot}^n$ converges in law to $X_{\cdot}$. Fix $x \in D$. By Skorokhod's representation theorem, there exist a probability space $(\widehat{\Omega},\widehat{P}_x)$ and a sequence of diffusion processes $\{\widehat{X}^n\}_{n \ge 1}$  and $\widehat{X}$, which have the same distributions as $\{X^n\}_{n \ge 1}$ and $X$, such that $\widehat{X}^n$ converges to $\widehat{X}$ in the space $C([0,\infty),\Rd)$ almost surely $w.r.t.$ $\widehat{P}_x$.

We still use $\tau_D$ and $\tau^n_D$ to denote the first exit time of $\widehat{X}$ and $\widehat{X}^n$ respectively. Since for any Borel measurable $A \subset \partial D$, $\{f\in C([0,\infty),\Rd): f(e_f) \in A\} \in \mathcal{B}(C([0,\infty),\Rd))$, where $e_f:=\inf\{t > 0:f(t) \notin D\}$, we have $E_x[\varphi(X_{\tau_D})]=\widehat E_x[\varphi(\widehat X_{\tau_D})]$, where $\widehat E_x$ is the expectation $w.r.t.$ $\widehat P_x$. The same is true for $X^n$ and $\widehat{X}^n$, $i.e.$ $E_x[\varphi(X^n_{\tau_D^n})]= \widehat E_x[\varphi(\widehat X^n_{\tau_D^n})]$. Hence, to prove (\ref{e3.24}) is the same as to show that

\begin{eqnarray}\label{e3.1}
\begin{split}
\lim_{n \to \infty}\widehat E_x[\varphi(\widehat X^n_{\tau_D^n})]=\widehat E_x[\varphi(\widehat X_{\tau_D})].
\end{split}
\end{eqnarray}

Since the transition density functions of $\widehat{X}_t$ and $\widehat{X}^n_t$ have Gaussian type two-sided estimates, by Lemma \ref{L3.8}, we have
\begin{eqnarray}\label{e3.2} \nonumber
\begin{split}
\widehat{P}_x(\tau_{\overline{D}}<\infty)=\widehat{P}_x(\tau_{\overline{D}}^n<\infty)=1.
\end{split}
\end{eqnarray}

For $l\ge 1$, set $C_l:=\bigcup_{n \ge 1}\{\omega \in \widehat{\Omega}: \tau_D^n(\omega)>l\}$. Then
\begin{eqnarray}\label{3.11} \nonumber
\begin{split}
& \ \ \ \ \bigcap_{l\ge 1}C_l=\{\omega \in \widehat{\Omega}: \exists \ \mbox{a sequence of} \ n_k\uparrow \infty,\ s.t. \  \tau^{n_k}_D(\omega)>k\}, \ P_x-a.e..
\end{split}
\end{eqnarray}
For any $\omega \in \{\tau_{\overline{D}}<\infty\}\bigcap \{\widehat{X}^n \to \widehat{X}\}$, there exists a constant $T(\omega)< \infty$ such that $\widehat X(\omega,T(\omega)) \in \overline{D}^c$ and hence we can find an integer $N(\omega)>0$, such that $\widehat X^{n}(\omega,T(\omega)) \in \overline{D}^c$ for any $n \ge N(\omega)$. Thus, we have $\tau^{n}_D(\omega)< T(\omega)$ for $n \ge N(\omega)$. Therefore $\{\tau_{\overline{D}}<\infty\}\bigcap \{\widehat{X}^n \to \widehat{X}\} \subset (\bigcap_{l\ge 1}C_l)^c$. Consequently,
\begin{eqnarray}\label{3.12}
\begin{split}
\widehat{P}_x(\bigcap_{l\ge 1}C_l)=0.
\end{split}
\end{eqnarray}

Set
$\widehat{\Omega}_l:=\bigcap_{k \ge 1}\{\tau_D^k \le l\} \bigcap \{\widehat{X}^n \to \widehat{X} \}\bigcap \{\tau_D=\tau_{\overline{D}}<\infty\}$.
Fix $\omega \in \widehat{\Omega}_l$. Since $\widehat{X}^n$ converges uniformly to $\widehat{X}$ over $[0,l]$ and $\widehat{X}^n_{\tau_D^n}\in \partial D$, we have $\widehat{X}_{\tau}\in \partial D$ for any limit point $\tau$ of $\{\tau_D^n\}_{n\ge 1}$ and therefore $\tau_D \le \tau$. On the other hand, for any $c>0$, since $\tau_{\overline{D}}(\omega)=\tau_D(\omega)$, there exists a $t< \tau_D+c$ such that $\widehat{X}_t \in \overline{D}^c$. Thus, $\widehat{X}^n_t\in \overline{D}^c$ for any $n$ sufficiently large. Hence $\tau_D^n < t<\tau_D+c$, which yields $\tau < \tau_D+c$. Let $c\to 0$ to deduce that $\tau=\tau_D$. Thus we have proven that $\lim_{n \to \infty} \tau_D^n(\omega)=\tau_D(\omega)$ for $\omega \in \widehat{\Omega}_l$. Since $\widehat{\Omega}_l \subset \bigcap_{n \ge 1}\{\tau_D^n \le l\}$, by the uniform convergence of $\widehat{X}^n$ over finite intervals, we deduce that
\begin{eqnarray}\label{3.14} \nonumber
\begin{split}
\lim_{n \to \infty} \widehat X^n_{\tau_D^n}(\omega)=\widehat X_{\tau_D}(\omega), \ \forall \omega \in \widehat{\Omega}_l.
\end{split}
\end{eqnarray}

On the other hand, by the strong Markov property of $\hat{X}$ and Lemma \ref{L3.8},
\begin{eqnarray}\label{4.14} \nonumber
\begin{split}
\widehat{P}_x(\tau_{\overline{D}} \neq \tau_D )=\widehat{P}_x(\tau_{\overline{D}} > \tau_D )=\widehat{E}_x[\widehat{P}_{\widehat{X}_{\tau_D}}(\tau_{\overline{D}} > 0) ]=0.
\end{split}
\end{eqnarray}
Together with (\ref{3.12}), we conclude that $\widehat P_x(\bigcup_{l \ge 1} \widehat{\Omega}_l)=1$. Hence
\begin{eqnarray}\label{4.12} \nonumber
\begin{split}
\lim_{n \to \infty} \widehat X^n_{\tau_D^n}(\omega)=\widehat X_{\tau_D}(\omega), \ \widehat P_x-a.e..
\end{split}
\end{eqnarray}
which yields (\ref{e3.1}) due to the continuity of $\varphi$.
\end{proof}

\vskip 0.4cm

It is well known that $\tilde{u}_n$ is the solution of the following equation:
\begin{align} \label{4.19}
\left\{
\begin{aligned}
& \frac{1}{2}\Delta \tilde{u}_n(x)+ \nabla \tilde{u}_n(x) \cdot  G_n(x) =0 ,&&\forall x \in D, \\
& \tilde{u}_n(x)|_{\partial D}=\varphi(x), && \forall x \in \partial D. \\
\end{aligned}
\right.
\end{align}
From the smoothness of $G_n$ and Theorem 6.17 in \cite{Gilbarg} we know that $\tilde{u}_n \in C^{\infty}(D)$.

We have the following Lemma.

\vskip 0.3cm

\begin{lem} \label{L4.3}
Assume $\varphi \in C^{1,\a_0}(\partial D)$. Then $\tilde{u} \in C^{1}_b(D)\bigcap C(\overline{D})$ and for any compact subset $K \subset D$,
\begin{eqnarray}\label{4.22}
\begin{split}
\lim_{n\to \infty}\sup_{x \in K} |\nabla \tilde{u}_n-\nabla \tilde{u}(x)|=0.
\end{split}
\end{eqnarray}
Moreover, if $\mu_i \in \K_{d,1-\a_0}$ for each $1 \le i \le d$, then $\tilde{u} \in C^{1,\a_0}(\overline{D})$ and there exists a $c>0$, depending on $\mu$ only via the function $\max_{1 \le i \le d}M_{\mu_i}^{1-\a_0}(\cdot)$, such that
\begin{eqnarray} \label{4.23}
\begin{split}
||\tilde u||_{C^{1,\a_0}(\overline{D})}\le c||\varphi||_{C^{1,\a_0}(\partial D)}.
\end{split}
\end{eqnarray}
\end{lem}

\begin{proof}
Recall that $p^{n,D}(t,x,y)$ is the transition density function of the killed process $X^{n,D}$.

Note that $\tilde{u}_0(x)=E_x[\varphi(W_{\tau_D})]$ is the solution of the Laplace's equation:
\begin{align} \nonumber
\left\{
\begin{aligned}
& \frac{1}{2}\Delta \tilde{u}_0(x)=0 ,&&\forall x \in D, \\
& \tilde{u}(x)_0|_{\partial D}=\varphi(x), && \forall x \in \partial D. \\
\end{aligned}
\right.
\end{align}
Since $\varphi \in C^{1,\a_0}(\partial D)$, it follows from Theorem 8.34 in \cite{Gilbarg} that
\begin{align}\label{3.20}
||\tilde{u}_0||_{C^{1,\a_0}(\overline{D})} \le c_1 ||\varphi||_{C^{1,\a_0}(\partial D)}<\infty.
\end{align}
for some constant $c_1>0$.

\vskip 0.3cm
First we prove that for any compact subset $K \subset D$,
\begin{eqnarray}\label{3.35}
\begin{split}
&\lim_{n \to \infty}\sup_{x\in K}|\int_0^{\infty}\int_D \partial_{x_i} p^{D}(s,x,y)\nabla \tilde{u}_0(y) \cdot \mu (dy) ds\\
&\ \ \ \ \ \ \ \ \ \ \ \ \ -\int_0^{\infty}\int_D \partial_{x_i} p^{n,D}(s,x,y)\nabla \tilde{u}_0(y)\cdot G_n(y) dy ds|\\
& =0.\\
\end{split}
\end{eqnarray}

Hereafter, set $|\mu|:=\sum_{1 \le i \le d} |\mu^i|$. For $0<\d<2<T$, $1 \le i \le d$ and a compact subset $D_c \subset D$ we have
\begin{eqnarray}\label{3.25}
\begin{split}
& \ \ \ \ \sup_{x\in K}|\int_0^{\infty}\int_D \partial_{x_i} p^{D}(s,x,y)\nabla \tilde{u}_0(y) \cdot \mu (dy) ds-\int_0^{\infty}\int_D \partial_{x_i} p^{n,D}(s,x,y)\nabla \tilde{u}_0(y)\cdot G_n(y) dy ds|\\
& \le   \sup_{x\in K}\int_T^{\infty}\int_D |\nabla \tilde{u}_0(y)| \{|\partial_{x_i} p^{D}(s,x,y)|   |\mu| (dy)+ |\partial_{x_i} p^{n,D}(s,x,y)| |G_n(y)|dy\}  ds\\
& \ \ \ \ +\sup_{x\in K}\int_0^{\d}\int_D |\nabla \tilde{u}_0(y)| \{|\partial_{x_i} p^{D}(s,x,y)|   |\mu| (dy)+ |\partial_{x_i} p^{n,D}(s,x,y)| |G_n(y)|dy\}  ds\\
& \ \ \ \ + \sup_{x\in K}\int_{\d}^{T}\int_{D\setminus D_{c}} |\nabla \tilde{u}_0(y)| \{|\partial_{x_i} p^{D}(s,x,y)|   |\mu| (dy)+ |\partial_{x_i} p^{n,D}(s,x,y)| |G_n(y)|dy\}  ds\\
& \ \ \ \ + \sup_{x\in K}|\int_{\d}^{T}\int_{D_c} \partial_{x_i} p^{D}(s,x,y)\nabla \tilde{u}_0(y) \cdot \mu (dy) ds-\int_{\d}^{T}\int_{D_c}  \partial_{x_i} p^{D}(s,x,y)\nabla \tilde{u}_0(y)\cdot G_n(y) dy ds|\\
&\ \ \ \ + \sup_{x\in K}|\int_{\d}^{T}\int_{D_c}  (\partial_{x_i} p^{D}(s,x,y)-\partial_{x_i} p^{n,D}(s,x,y))\nabla \tilde{u}_0(y)\cdot G_n(y) dy ds|\\
&:=I_1^{n}+I_2^{n}+I_3^{n}+I_4^{n}+I_5^{n}.
\end{split}
\end{eqnarray}

By (\ref{3.59}) and Fubini's theorem, for any $\e>0$ there exists a $T>2$ such that
\begin{eqnarray}\label{3.26}
\begin{split}
\sup_{n \ge 1} I_1^{n} &\le \sup_{y \in D} |\nabla \tilde{u}_0(y)|  \int_T^{\infty} M_{11}e^{-M_8(1)s}ds\{|\mu|(D)+\sup_{n \ge 1} \sup_{x \in \Rd}\int_D\int_{\Rd}\psi_n(y-x)|\mu|(dx)dy\}\\
& =\sup_{y \in D} |\nabla \tilde{u}_0(y)| 2M_{11}M_8(1)^{-1}e^{-M_8(1)T} |\mu|(D)\le \e. \\
\end{split}
\end{eqnarray}

In view of (\ref{2.10}) and ({\ref{3.38}}), there exists a $\d <2$ such that
\begin{eqnarray}\label{3.28}
\begin{split}
\sup_{n \ge 1}I_2^{n} &\le \sup_{y \in D} |\nabla \tilde{u}_0(y)| \sup_{n \ge 1}\sup_{x \in K} \int_0^{\d}\int_{\Rd}M_{12}s^{-\frac{d+1}{2}}e^{-\frac{M_{10}|x-y|^2}{2s}}(|\mu| (dy)+|G_n(y)|dy)  ds\\
&= \sup_{y \in D} |\nabla \tilde{u}_0(y)| M_{12}  \{N_{|\mu|}^{1,M_{10}}(\d)+\sup_{n \ge 1}\sum_{1 \le i \le d} N_{G_n^i}^{1,M_{10}}(\d)\}\le \e.\\
\end{split}
\end{eqnarray}

For fixed $T$ and $\d$, there exists a compact subset $D_{c} \subset D$ with a smooth boundary such that
$(T-\d)M_{12}{\d}^{-\frac{d+1}{2}}  |\mu| (D\setminus D_{c})<\e$. Note that $|\mu|(\partial D)=0$ and $|\mu|(\partial D_c)=0$ by Lemma 3.2 in \cite{KimSong}, thus,
\begin{eqnarray*}
\begin{split}
&\ \ \ \ (T-\d)M_{12}{\d}^{-\frac{d+1}{2}} (|\mu| (D\setminus D_{c})+ \lim_{n \to \infty}\int_{D\setminus D_c}  |G_n(y)|dy)\\
&\le 2(T-\d)M_{12}{\d}^{-\frac{d+1}{2}}  |\mu| (D\setminus D_{c})<2\e.
\end{split}
\end{eqnarray*}
Therefore, there exists a positive integer $N$ such that
\begin{eqnarray}\label{3.31}
\begin{split}
\sup_{n \ge N} I_3^{n} &\le \sup_{y \in D} |\nabla \tilde{u}_0(y)|  \int_{\d}^{T} M_{12}s^{-\frac{d+1}{2}}e^{-\frac{M_{10}|x-y|^2}{2s}}ds \sup_{n \ge N}  \int_{D\setminus D_c} ( |\mu| (dy)+ |G_n(y)|dy) \\
&\le \sup_{y \in D} |\nabla \tilde{u}_0(y)| (T-\d)M_{12}\d^{-\frac{d+1}{2}}\sup_{n \ge N}\{|\mu| (D\setminus D_c)+\int_{D\setminus D_c}  |G_n(y)|dy \} \\
&\le 2\e \sup_{y \in D} |\nabla \tilde{u}_0(y)|. \\
\end{split}
\end{eqnarray}

Since $\partial_{x_i} p^{D}(s,x,y)$ is continuous in $[\d,T] \times K \times D_c$ by Theorem 4.6 in \cite{KimSong}, $g(x,y):=\int_{\d}^{T} \partial_{x_i} p^{D}(s,x,y)ds \nabla \tilde{u}_0(y)$ is also continuous in $K \times D_c$. It follows from Lemma 3.3 in \cite{KimSong} that
\begin{eqnarray}\label{3.32}
\begin{split}
\lim_{n \to \infty}I_4^n =\lim_{n \to \infty} \sup_{x \in K} |\int_{D_c}g(x,y)\cdot G_n(y)dy-\int_{D_c}g(x,y)\cdot \mu (dy)|=0.
\end{split}
\end{eqnarray}

By Theorem 4.5 in \cite{KimSong}, we have
\begin{eqnarray}\label{3.33}
\begin{split}
\lim_{n \to \infty}I_5^n \le & \lim_{n \to \infty} \sup_{(s,x,y)\in [\d,T] \times K \times D_c} |\partial_{x_i} p^{D}(s,x,y)-\partial_{x_i} p^{n,D}(s,x,y)| \\
& \times (T-\d) \sup_{y\in D_c}|\nabla \tilde{u}_0(y)|  \sup_{n \ge 1}\int_{D}|G_n(y)| dy \\
&=0.\\
\end{split}
\end{eqnarray}

Since $\e$ is arbitrary, we deduce (\ref{3.35}) from (\ref{3.25})-(\ref{3.33}).

\vskip 0.3cm

Next we show that for any compact subset $K \subset D$,
\begin{align}\label{e3.3}
\lim_{n,m\to \infty} \sup_{x\in K}|\nabla \tilde{u}_n(x)-\nabla \tilde{u}_m(x)|=0.
\end{align}

Applying Ito's formula to $\tilde{u}_0(X^n_t)$, we get that for $x \in D$,
\begin{align}\label{3.18}
\begin{split}
\tilde{u}_0(X^n_t)&=\tilde{u}_0(x)+\int_0^t  \nabla \tilde{u}_0(X^n_s) \cdot dW_s +\int_0^t  \nabla \tilde{u}_0(X^n_s) \cdot G_n(X_s^n) ds, \ \forall t<\tau_D^n.
\end{split}
\end{align}
Remembering that $\nabla \tilde{u}_0$ is bounded on $\overline{D}$ (see (\ref{3.20})) and $\tilde u_0= \varphi$ on $\partial D$, we obtain from (\ref{3.18})
\begin{eqnarray}\label{3.19} \nonumber
\begin{split}
\tilde{u}_n(x)&=E_x[\tilde{u}_0(X^n_{\tau_D^n})]\\
&=\tilde{u}_0(x)+E_x[\int_0^{\tau_D^n} \nabla \tilde{u}_0(X^n_s) \cdot G_n(X_s^n) ds]\\
&=\tilde{u}_0(x)+E_x[\int_0^{\tau_D^n} \nabla \tilde{u}_0(X^n_s) \cdot G_n(X_s^n) ds]\\
&=\tilde{u}_0(x)+\int_0^{\infty}\int_D  p^{n,D}(s,x,y)\nabla \tilde{u}_0(y) \cdot G_n(y)  dy ds.\\
\end{split}
\end{eqnarray}
This implies, in view of (\ref{3.38}), that for $1 \le i \le d$ and $n \ge 1$,
\begin{eqnarray}\label{3.23}
\begin{split}
\partial_{x_i} \tilde{u}_n(x)=\partial_{x_i} \tilde{u}_0(x)+\int_0^{\infty}\int_D \partial_{x_i} p^{n,D}(s,x,y)\nabla \tilde{u}_0(y) \cdot G_n(y) dy ds, \ \forall x\in D.\\
\end{split}
\end{eqnarray}
Combining (\ref{3.23}) with (\ref{3.35}) we prove (\ref{e3.3}).

\vskip 0.3cm

Using Lemma \ref{L4.1}, Lemma \ref{Lemma 3.2} and (\ref{e3.3}) we conclude that $u\in C^1(D)\bigcap C(\overline{D})$ and $\nabla \tilde{u}_n$ uniformly converges to $\nabla \tilde{u}$ over compact subsets of $D$. Taking into accounts (\ref{3.35}) and letting $n \to \infty$ in (\ref{3.23}), we obtain
\begin{align}\label{3.36}
\partial_{x_i} \tilde{u}(x)=\partial_{x_i} \tilde{u}_0(x)+\int_0^{\infty}\int_D  \partial_{x_i} p^{D}(s,x,y)\nabla \tilde{u}_0(y)\cdot \mu(dy) ds,\ \forall x \in D.
\end{align}
By (\ref{3.38}), (\ref{3.20}) and Lemma \ref{L3.7}, this in particular implies that $\partial_{x_i} \tilde{u}(x)$ is bounded in $D$.

\vskip 0.3cm

If $\mu_i \in \K_{d, 1-\a_0}$ for each $1 \le i \le d$, by Proposition \ref{p3.4}, Lemma \ref{L3.7}, (\ref{3.20}) and (\ref{3.36}), it is easy to see that $u\in C^{1,\a_0}(\overline{D})$ and there exists a $c_2>0$, depending on $\mu$ only via the function $\max_{1 \le i \le d}M_{\mu_i}^{1-\a_0}(\cdot)$, such that for any convex subset $D' \subset D$, $x_0 \in D$ and $x,x' \in \overline{D'}$,

\begin{eqnarray} \nonumber
\begin{split}
&\ \ \ \ |\partial_{x_i} \tilde{u}(x)-\partial_{x_i} \tilde{u}(x')|\\
& \le \int_0^{\infty}\int_D |\partial_{x_i} p^{D}(t,x,y)-\partial_{x_i} p^{D}(t,x',y)| |\nabla \tilde{u}_0(y)||\mu|(dy) dt+|\partial_{x_i} \tilde{u}_0(x)-\partial_{x_i} \tilde{u}_0(x')|\\
& \le M_{14}|x-x'|^{\a_0}\sup_{y \in D}|\nabla \tilde{u}_0(y)| \int_{B_{R_0}(x_0)} \int_0^{\infty} t^{-\frac{d+1+\a_0}{2}}(e^{-\frac{M_{15}|x-y|^2}{2t}}+e^{-\frac{M_{15}|x'-y|^2}{2t}})dt|\mu|(dy) \\
& \ \ \ \ +|\partial_{x_i} \tilde{u}_0(x)-\partial_{x_i} \tilde{u}_0(x')|\\
& \le M_{14}|x-x'|^{\a_0}c_1 ||\varphi||_{C^{1,\a_0}(\partial D)}c_3\sup_{x \in \overline{D}}\int_{B_{R_0}(x_0)}|x-y|^{1-\a_0-d}|\mu|(dy)+c_1 |x-x'|^{\a_0}||\varphi||_{C^{1,\a_0}(\partial D)}\\
& \le c_2 |x-x'|^{\a_0}||\varphi||_{C^{1,\a_0}(\partial D)},
\end{split}
\end{eqnarray}
where $R_0$ is the diameter of $D$ and $c_3>0$ is some constant depending only on $d,\a_0$ and $M_{15}$.
Therefore, combining with (\ref{3.38}), (\ref{3.39}) and Lemma \ref{L3.7}, (\ref{4.23}) holds.
\end{proof}

\vskip 0.4cm

Let $V_t$ denote the CAF associated with $\p$. Set $\hat{u}(x):=E_x[V_{\tau_D}]$, $K_n(x):=\int_{\Rd}\psi_n(x-y)\p (dy)$ and $\hat{u}_n(x):=E_x[\int_0^{\tau_D^n}K_n(X_s^n)ds]$. It is well known that $\hat{u}_n$ is the solution of the following equation:
\begin{align} \label{4.20}
\left\{
\begin{aligned}
& \frac{1}{2}\Delta \hat{u}_n(x)+ \nabla \hat{u}_n(x) \cdot  G_n(x) =-K_n(x) ,&&\forall x \in D, \\
& \hat{u}_n(x)|_{\partial D}=0, && \forall x \in \partial D. \\
\end{aligned}
\right.
\end{align}
Due to the smoothness of $G_n,K_n$ and Theorem 6.17 in \cite{Gilbarg} we know that $\hat{u}_n \in C^{\infty}(D)$.

\vskip 0.3cm
We have the following result.

\begin{lem} \label{L4.4}
$\hat{u} \in C(\overline{D})\bigcap C^{1}_b(D)$ and for any compact subset $K \subset D$,
\begin{eqnarray}\label{4.1}
\begin{split}
\lim_{n\to \infty}\sup_{x \in K} |\nabla \hat{u}_n-\nabla \hat{u}(x)|=0.
\end{split}
\end{eqnarray}
Moreover if for each $1 \le i \le d$, $\mu_i,\p \in \K_{d,1-\a_0}$, then $\hat{u} \in C^{1,\a_0}(\overline{D})$ and there exists a $c>0$, depending on $\mu$ only via the function $\max_{1 \le i \le d}M_{\mu_i}^{1-\a_0}(\cdot)$, such that
\begin{eqnarray} \label{4.24}
\begin{split}
||\hat u||_{C^{1,\a_0}(\overline{D})}\le cM_{\p}^{1-\a_0}(R_0).
\end{split}
\end{eqnarray}
\end{lem}

\begin{proof}
Fix $x_0 \in \partial D$, we have $\hat u(x_0)=0$ by Lemma \ref{L3.8}.
We now show that
\begin{eqnarray}\label{4.18}
\begin{split}
\lim_{D\ni x \to x_0}\hat u(x)=0.
\end{split}
\end{eqnarray}

By (\ref{3.43}),
\begin{eqnarray}\label{4.15}
\begin{split}
\hat{u}(x)=E_x[V_{\tau_D}]=\int_0^{\infty}\int_Dp^D(s,x,y)\rho(dy)ds.
\end{split}
\end{eqnarray}
Using (\ref{3.39}), it is easy to see that
\begin{eqnarray}\label{4.3} \nonumber
\begin{split}
\sup_{x \in D}|E_{x}[V_{\tau_D}]|=\sup_{x \in D}|\int_0^{\infty}\int_{D}p^D(s,x,y)\p(dy)ds|<\infty.\\
\end{split}
\end{eqnarray}
Noting that $\tau_D=t+\tau_D \circ \theta_t$ on $\{\omega: t < \tau_D\}$, for $x \in D$ and $t>0$, we have
\begin{eqnarray}\label{4.2}
\begin{split}
|\hat{u}(x)|&= |E_x[V_{\tau_D}; t< \tau_D]+E_x[V_{\tau_D}; t\ge \tau_D]|\\
&= |E_x[V_{\tau_D}\circ \theta_t+V_{t}; t< \tau_D]+E_x[V_{\tau_D}; t\ge \tau_D]|\\
&\le |E_x[E_x[V_{\tau_D}\circ \theta_t|\mathcal{F}_{t}]+V_{t}; t< \tau_D]|+E_x[|V|_{\tau_D}; t\ge \tau_D]\\
&\le |E_x[E_{X_{t }}[V_{\tau_D}]+V_{t}; t< \tau_D]|+E_x[|V|_{t}; t\ge \tau_D]\\
&\le |E_x[E_{X_{t}}[V_{\tau_D}]; t< \tau_D]|+E_x[|V|_{t}; t< \tau_D]+E_x[|V|_{t}; t\ge \tau_D]\\
&\le P_x[t< \tau_D]\sup_{x \in D}|E_{x}[V_{\tau_D}]|+E_x[|V|_{t}].\\
\end{split}
\end{eqnarray}
By (\ref{3.39}) and (\ref{3.44}), we have
\begin{eqnarray}\label{4.4} \nonumber
\begin{split}
\lim_{t \to 0}\sup_{x \in D} E_{x}[|V|_{t}] =\lim_{t \to 0}\sup_{x \in \Rd} \int_0^{t}\int_{D}p^D(t,x,y)|\p|(dy)dt =0.\\
\end{split}
\end{eqnarray}
This together with (\ref{3.3}), (\ref{4.3}) and (\ref{4.2}) implies (\ref{4.18}).

\vskip 0.3cm

On the other hand, similar to (\ref{3.35}), we have for any compact subset $K \subset D$ and $1 \le i \le d$,
\begin{eqnarray}\label{4.16}
\begin{split}
&\lim_{n \to \infty}\sup_{x\in K}|\int_0^{\infty}\int_D  p^{D}(s,x,y) \p (dy) ds-\int_0^{\infty}\int_D p^{n,D}(s,x,y)K_n(y) dy ds|=0,\\
&\lim_{n \to \infty}\sup_{x\in K}|\int_0^{\infty}\int_D \partial_{x_i} p^{D}(s,x,y)\p (dy) ds-\int_0^{\infty}\int_D \partial_{x_i} p^{n,D}(s,x,y)K_n(y) dy ds|=0.\\
\end{split}
\end{eqnarray}
Since $\hat{u}_n(x)=\int_0^{\infty}\int_D  p^{n,D}(s,x,y)K_n(y) dy ds \in C^{1}_b(D)$, combining (\ref{4.15}) with (\ref{4.16}), we prove both $\hat{u} \in C^{1}_b(D)$ and (\ref{4.1}).

Moreover if for each $1 \le i \le d$, $\mu_i,\p \in \K_{d,1-\a_0}$, then by Proposition \ref{p3.4}, Lemma \ref{L3.7}, (\ref{3.38}), (\ref{3.39}) and (\ref{4.15}), it is easy to see that $\hat{u} \in C^{1,\a_0}(\overline{D})$ and (\ref{4.24}).
\end{proof}

\vskip 0.4cm

% Note that  we arrive at

Now we are ready to give the following existence result.

\begin{thm}\label{T4.1}
Assume $\varphi \in C^{1,\a_0}(\partial D)$. Let $u(x):=E_x[\varphi(X_{\tau_D})+V_{\tau_D}]$. Then $u \in C_b^1(D)$ is a weak solution to Dirichlet boundary problem (\ref{3.1}).
%And the solution $u$ is given by

Moreover, if for each $1 \le i \le d$, $\mu_i,\p \in \K_{d,1-\a_0}$, then $u \in C^{1,\a_0}(\overline{D})$ and there exists a $c>0$, depending on $\mu$ only via the function $\max_{1 \le i \le d}M_{\mu_i}^{1-\a_0}(\cdot)$, such that
\begin{eqnarray} \label{4.25}
\begin{split}
||u||_{C^{1,\a_0}(\overline{D})}\le c(||\varphi||_{C^{1,\a_0}(\partial D)}+M_{\p}^{1-\a_0}(R_0)).
\end{split}
\end{eqnarray}
\end{thm}

\begin{proof}
It is easy to see that $u=\tilde u + \hat u \in C(\overline{D})\bigcap C^1_b(D)$ and $u(x)= \varphi(x)$ on $\partial D$ by Lemma \ref{L4.3} and Lemma \ref{L4.4}. Set $u_n(x):=\tilde u_n(x)+\hat{u}_n(x)$, which are the solutions of the Dirichlet problems (\ref{4.19}) and (\ref{4.20}) respectively. Hence for any $\phi \in C_0^{\infty}(D)$,
\begin{eqnarray}\label{4.21}
\begin{split}
\frac{1}{2}\sum_{i=1}^{d}\int_{D}\frac{\partial u_n(x)}{\partial x_i}\frac{\partial \phi(x)}{\partial x_i}dx-\sum_{i=1}^d\int_{D}\frac{\partial u_n(x)}{\partial x_i}\phi(x) \cdot G_n^i(x)dx=\int_D \phi(x)K_n(x) dx.
\end{split}
\end{eqnarray}
Since $G_n(x)dx$ and $K_n(x)dx$ converge weakly to $\mu(dx)$ and $\p(dx)$, then combining (\ref{4.22}) and (\ref{4.1}), letting $n \to \infty$ in (\ref{4.21}), we obtain
\begin{eqnarray} \nonumber
\begin{split}
\frac{1}{2}\sum_{i=1}^{d}\int_{D}\frac{\partial u(x)}{\partial x_i}\frac{\partial \phi(x)}{\partial x_i}dx-\sum_{i=1}^d\int_{D}\frac{\partial u(x)}{\partial x_i}\phi(x) \mu_i(dx)=\int_D \phi(x)\p(dx),
\end{split}
\end{eqnarray}
which implies $u$ is a weak solution to Dirichlet boundary problem (\ref{3.1}). Moreover by Lemma \ref{L4.3} and Lemma \ref{L4.4}, if for each $1 \le i \le d$, $\mu_i,\p \in \K_{d,1-\a_0}$, then $u \in C^{1,\a_0}(\overline{D})$ and (\ref{4.25}) holds.
\end{proof}

% Moreover, Under the Assumption I, the solution to problem (\ref{2.2}) is unique such that $u \in C^{1,\a_0}(D)$.

% Let $u(x)$ be defined as the right hand side of (\ref{4.5}). Since $\tilde{u}_n=E_x[\tilde{u}_0(X^n_{\tau_D^n})]$ is the solution of problem (\ref{3.16}).

\vskip 0.4cm

Now we turn to the uniqueness of the solution of the problem (\ref{3.1}).

\begin{lem} \label{L4.5}
Let $B_r:=B_{x_0}(r)\subset D$, where $r \le r_0$ and $r_0$ is defined in Lemma \ref{L3.5}.  Assume $\tilde{\varphi} \in C^{1,\a_0}(\partial B_r)$. Then there exists a unique weak solution $u$ to the following equation
\begin{align} \label{4.5}
\left\{
\begin{aligned}
& \frac{1}{2}\Delta u(x)+ \nabla u(x) \cdot  \mu(dx) =-\p(dx),&&\forall x \in B_r, \\
& u(x)|_{\partial B_r}=\tilde{\varphi}(x), && \forall x \in \partial B_r, \\
\end{aligned}
\right.
\end{align}
satisfying $u \in C^{1}_b(B_r)$. Moreover $u$ is given by
\begin{eqnarray} \nonumber
\begin{split}
u(x)=E_x[\tilde{\varphi}(X_{\tau_{B_r}})+V_{\tau_{B_r}}].
\end{split}
\end{eqnarray}
\end{lem}

\begin{proof}
By Lemma 3.2 in \cite{KimSong} we know that $|\mu|(\partial B_r)=0$, hence existence of a solution to the problem (\ref{4.5}) was given in Theorem \ref{T4.1}. We just need to prove the uniqueness of the solution.

Assume $u_1,u_2 \in C^{1}_b(B_r)$ are two weak solutions to problem (\ref{4.5}). Let $v=u_1-u_2$. Then
\begin{align} \nonumber
\left\{
\begin{aligned}
& \frac{1}{2}\Delta v(x)=- \nabla v(x) \cdot  \mu(dx),&&\forall x \in B_r, \\
& v(x)|_{\partial B_r}=0, && \forall x \in \partial B_r. \\
\end{aligned}
\right.
\end{align}

\vskip 0.3cm
We aim to prove $v=0$.

Set the measure $\mu_{B_r}(dx):=I_{B_r}(x)\mu(dx)$. By Remark \ref{R3.1} and Proposition \ref{P3.2}, it is seen that on $\{t \le \tau_{B_r}\}$, $\int_0^t  \nabla v(X_s)\cdot dA_s$ is the CAF associated with the measure $\nabla v(x) \cdot  \mu_{B_r}(dx)$. Hence by Proposition \ref{P2.2} and Theorem \ref{T4.1} we know that $v_1(x)=\int_0^{\infty}\int_{B_r}q^{B_r}(s,x,y)\nabla v(y)\cdot \mu_{B_r}(dy)ds$ is the weak solutions to the following equation:
\begin{align} \nonumber
\left\{
\begin{aligned}
& \frac{1}{2}\Delta v_1(x)=- \nabla v(x) \cdot  \mu_{B_r}(dx),&&\forall x \in B_r, \\
& v(x)|_{\partial B_r}=0, && \forall x \in \partial B_r, \\
\end{aligned}
\right.
\end{align}
where $q^{B_r}(s,x,y)$ denotes the transition density function of the killed process $W^{B_r}$.
According to the uniqueness of the solution to the above equation, we conclude that $v(x)=v_1(x)=\int_0^{\infty}\int_{B_r}q^{B_r}(s,x,y)\nabla v(y)\cdot \mu_{B_r}(dy)ds$ in $B_r$.

For $\pi \in \K_{d,1}$, define $R\pi(x):=\int_0^{\infty}\int_{B_r}q^{B_r}(s,x,y) \pi(dy)dt$ and $\mathfrak{B} \pi:=\nabla R \pi \cdot \mu_{B_r}$. Then for $x \in B_r$,
\begin{eqnarray} \nonumber
\begin{split}
v(x)&=R(\nabla v \cdot \mu_{B_r})(x)\\
&=R(\nabla R(\nabla v \cdot \mu_{B_r}) \cdot \mu_{B_r})(x)\\
&=R(\B (\nabla v \cdot \mu_{B_r}))(x)\\
&=R(\B (\nabla R(\nabla v \cdot \mu_{B_r}) \cdot \mu_{B_r}))(x)\\
&=R(\B ^2(\nabla v \cdot \mu_{B_r}))(x).\\
\end{split}
\end{eqnarray}
By iterating we obtain
$v(x)=R(\B^n (\nabla v \cdot \mu_{B_r}))(x)$ for any $n \ge 1$.
By Lemma \ref{L3.5}, we have
\begin{eqnarray} \label{4.11}
\begin{split}
\sup_{x \in B_r}|\nabla v(x)|&=\sup_{x \in B_r}|\nabla R(\B^n (\nabla v \cdot \mu_{B_r}))(x)|\\
&\le M_{16}M^1_{\B^n (\nabla v \cdot \mu_{B_r})}(r)\\
&\le 2^{-n}\sup_{x \in B_r}|\nabla v(x)|M_{16}M^1_{\mu}(r).\\
\end{split}
\end{eqnarray}
Let $n \to \infty$ in (\ref{4.11}) to obtain
\begin{eqnarray} \nonumber
\begin{split}
|\nabla v(x)|=0, \ \forall x \in B_r.
\end{split}
\end{eqnarray}
Since $v \in C(\overline{B_r})$ and $v=0$ on $\partial B_r$, we conclude that $v=0$ in $B_r$, which is the uniqueness of the solution to the problem (\ref{4.5}).
\end{proof}

\vskip 0.4cm

Now we can state the main result in this section.

\begin{thm} \label{T4.2}
Assume $\mu,\p \in \K_{d,1-\a_0}$ and $\varphi \in C^{1,\a_0}(\partial D)$. Then there exists a unique weak solution $u$ in $C^{1,\a_0}(D)$ to problem (\ref{3.1}). Moreover the solution $u$ is given by
\begin{eqnarray} \nonumber
\begin{split}
u(x)=E_x[\varphi(X_{\tau_D})+V_{\tau_D}].
\end{split}
\end{eqnarray}
\end{thm}
% Moreover, Under the Assumption I, the solution to problem (\ref{2.2}) is unique such that $u \in C^{1,\a_0}(D)$.

\begin{proof}
Existence was proved in Theorem \ref{T4.1}. It remains to prove the uniqueness.

Suppose $u_1,u_2 \in C^{1,\a_0}(D)$ are two weak solutions to problem (\ref{3.1}). Let $v=u_1-u_2$. Take a sequence of balls $\{B_{r_n}(x_n)\}_{n \ge 1}$ that satisfies $\bigcup_{n \ge 1}B_{r_n}(x_n)=D$, $\overline{B_{r_n}(x_n)} \subset D$ and $r_n \le r_0$ for each $n \ge 1$, where $r_0$ is defined in Lemma \ref{L3.5}.

For any $n \ge 1$, set $B_n:=B_{r_n}(x_n)$. Note that $v$ satisfies the following equation:
\begin{align} \nonumber
\left\{
\begin{aligned}
& \frac{1}{2}\Delta v(x)+ \nabla v(x) \cdot  \mu(dx)=0,&&\forall x \in B_n , \\
& v(x)|_{\partial B_n}=v(x), && \forall x \in \partial B_n . \\
\end{aligned}
\right.
\end{align}
Since $v \in C^{1,\a_0}(\partial B_{n})$, by Lemma \ref{L4.5} we have $v(x)=E_x[v(X_{\tau_{B_n}})]$ for $x \in \overline{B}_n$.

Let us now consider $B_1\bigcup B_2$.
Let $x \in B_1$ and define a sequence of stoping times $\tau_k$ recursively as follows:
\begin{eqnarray} \nonumber
\begin{split}
&\tau_0:=0,\\
&\tau_{2k-1}:=\inf \{t>\tau_{2k-2}:X_t \in B_1^c\},\\
&\tau_{2k}:=\inf \{t>\tau_{2k-1}:X_t \in B_2^c\}.\\
\end{split}
\end{eqnarray}
By the strong Markov property of $X$, we have
\begin{eqnarray} \nonumber
\begin{split}
E_x[v(X_{\tau_2})]=E_x[E_x[v(X_{\tau_2})|\mathcal{F}_{\tau_1}]]=E_x[E_{X_{\tau_1}}(v(X_{\tau_{B_2}}))]=E_x[v(X_{\tau_1})]=v(x).
\end{split}
\end{eqnarray}
Similarly, we get $E_x[v(X_{\tau_n})]=v(x)$ for $n \ge 1$.
On the other hand, since $X_t$ is continuous, we can show $\lim_{n \to \infty}\tau_n=\tau_{B_1 \cup B_2}< \infty$, $P_x$-$a.e.$. Hence, we have $v(x)=\lim_{n\to \infty}E_x[v(X_{\tau_n})]=E_x[v(X_{\tau_{B_1 \cup B_2}})]$.

Repeating the same procedure, we have $v(x)=E_x[v(X_{\tau_{\cup_{1 \le i \le n} B_i}})]$ for all $n \ge 1$. Since $\lim_{n\to \infty} \tau_{\cup_{1 \le i \le n} B_i}=\tau_D$, $P_x$-$a.e.$, it follows that
\begin{eqnarray} \nonumber
\begin{split}
v(x)=\lim_{n \to \infty}E_x[v(X_{\tau_{\cup_{1 \le i \le n} B_i}})]=E_x[v(X_{\tau_D})]=0,\ \forall x \in B_1.
\end{split}
\end{eqnarray}
The same will be valid for $x \in B_n$ for any fixed $n \ge 1$. Hence, $v(x)=0$ for all $x \in D$. The desired uniqueness is proved.
\end{proof}

\section{Case for $\nu \neq 0$}

We now consider the following problem:
\begin{align}\label{5.1}
\left\{
\begin{aligned}
&\frac{1}{2}\Delta u + \nabla u  \cdot  \mu+  u  \nu  =-\p,&&\forall x \in D, \\
& u(x)|_{\partial D}=\varphi(x), && \forall x \in \partial D. \\
\end{aligned}
\right.
\end{align}

Let $L_t$ denote the CAF associated with $\nu$. Recall that $V_t$ is the CAF associated with $\p$.

\begin{lem} \label{L5.1}
Assume that there exists a $x_0 \in D$ such that
\begin{eqnarray}\label{5.2}  \nonumber
\begin{split}
E_{x_0}[e^{L_{\tau_D}}]< \infty.
\end{split}
\end{eqnarray}
Then we have
\begin{eqnarray}
\sup_{x \in D}E_x[e^{L_{\tau_D}}]< \infty,\label{5.7}\\
\sup_{x \in D}E_x[\int_0^{\tau_D}e^{L_s}d|V|_s]< \infty.\label{5.15}
\end{eqnarray}
\end{lem}

\begin{proof}  For any bounded Borel measurable function $f$ on $D$, define
\begin{eqnarray}\label{5.31} \nonumber
Q_tf(x):=E_x[e^{L_{t}}f(X_t);t < \tau_D].
\end{eqnarray}
Then $Q_t,t > 0$, is a semigroup. We first show that $Q_t$ is strong Feller, $i.e.$ for any bounded Borel measurable function $f$ on $D$,
\begin{eqnarray}\label{5.16}
\begin{split}
Q_tf(x) \in C_b(D).
\end{split}
\end{eqnarray}

By the Markov property of $X^D$, for $0<s<t$,
\begin{eqnarray} \nonumber
\begin{split}
Q_tf(x)&=Q_s(Q_{t-s}f)(x)=E_x[Q_{t-s}f(X^D_s)]+E_x[(e^{L_s}-1)Q_{t-s}f(X^D_s)].
\end{split}
\end{eqnarray}
Hence
\begin{eqnarray}\label{5.11}
\begin{split}
|Q_tf(x)-E_x[Q_{t-s}f(X^D_s)]|\le ||f||_{L^{\infty}}E_x[(e^{L_s}-1)^2]^{\frac{1}{2}}\sup_{x \in D}E_x[e^{|L|_{t}}].
\end{split}
\end{eqnarray}
By (\ref{3.63}), we get
\begin{eqnarray}\label{5.12}
\begin{split}
&\ \ \ \ \lim_{s \to 0}\sup_{x \in \Rd} E_x[(e^{L_s}-1)^2]\\
& \le \lim_{s \to 0}\sup_{x \in \Rd} E_x[e^{2|L|_s}-1]\\
&=\lim_{s \to 0}\sup_{x \in \Rd} E_x[e^{2|L|_s}]-1=0.\\
\end{split}
\end{eqnarray}
It follows from (\ref{3.62}), (\ref{5.11}) and (\ref{5.12}) that
\begin{eqnarray}\label{5.26}
\begin{split}
\lim_{s \to 0}\sup_{x \in D}|Q_tf(x)-E_x[Q_{t-s}f(X^D_s)]|=0.
\end{split}
\end{eqnarray}

On the other hand, since $|Q_{t-s}f(x)|=|E_x[e^{L_{{t-s}}}f(X_{t-s});{t-s} < \tau_D]|\le ||f||_{L^{\infty}}|E_x[e^{L_{{t-s}}}]$ is bounded in $D$ by (\ref{3.62}), the strong Feller property of $X^D$ yields $E_x[Q_{t-s}f(X^D_s)] \in C_b(D)$. Together with (\ref{5.26}) we prove (\ref{5.16}).

\vskip 0.3cm

Now we prove (\ref{5.7}).
For any open subset $E \subset D$ and $s>0$, by the Markov property of $X_t$ we have
\begin{eqnarray} \label{5.3}
\begin{split}
\sup_{x \in \Rd}E_x[|L|_{\tau_E}]& \le \sup_{x \in \Rd} \sum_{n\ge 0}E_x[|L|_{(n+1)s}-|L|_{ns};\tau_E> ns]\\
& \le \sup_{x \in \Rd} \sum_{n\ge 0}E_x[E_{X_{ns}}[|L|_{s}];\tau_E> ns]\\
& \le \sup_{x \in \Rd} E_x[|L|_{s}]\sup_{x \in \Rd}\sum_{n\ge 0}P_x[\tau_E> ns].\\
\end{split}
\end{eqnarray}
By (\ref{3.44}), we have
\begin{eqnarray} \label{5.4}
\begin{split}
\lim_{s \to 0}\sup_{x \in \Rd} E_x[|L|_{s}]=\lim_{s \to 0} \sup_{x \in \Rd} \int_0^{s} \int_{\Rd}p(t,x,y)|\nu|(dy)dt =0.
\end{split}
\end{eqnarray}
On the other hand, by (\ref{3.39}), we have
\begin{eqnarray} \label{5.5}
\begin{split}
&\ \ \ \ \sup_{x \in \Rd}\sum_{n\ge 0}P_x[\tau_E> ns]\\
&\le 1+\sup_{x \in D}\sum_{n\ge 1}P_x[X_{ns}^D\in E]\\
& \le 1+\sup_{x \in D}\sum_{n\ge 1}\int_E M_{13}(ns)^{-\frac{d}{2}}e^{-\frac{M_{10}|x-y|^2}{2ns}}dy\\
&\le 1+M_{13}s^{-\frac{d}{2}}m(E) \sum_{n \ge 1}n^{-\frac{d}{2}}.\\
\end{split}
\end{eqnarray}
Hence by (\ref{5.3})-(\ref{5.5}), for any $\e>1$, there exists a $\d>0$ such that for any open subset $E \subset D$ with $m(E)<\d$, we have
$\sup_{x \in E}E_x[|L|_{\tau_E}]\le 1-\frac{1}{\e}$.
This together with Lemma \ref{L3.6} gives
\begin{eqnarray} \label{5.35}
\begin{split}
\sup_{x \in E}E_x[e^{|L|_{\tau_E}}]\le \e.
\end{split}
\end{eqnarray}

Using (\ref{5.35}) and the strong Feller property of $Q_t,t>0$, following the same arguments as in the section 5.6 in \cite{ChungZhao} we obtain (\ref{5.7}).

\vskip 0.3cm

Next we show that
\begin{eqnarray}
\begin{split}
\sup_{x \in D}E_x[\int_0^{\tau_D}e^{L_s}ds]< \infty.\label{5.14}\\
\end{split}
\end{eqnarray}

By Lemma \ref{L3.8} and (\ref{3.43}), we have
\begin{eqnarray}\label{5.22}
\begin{split}
\inf_{x \in \overline{D}}E_x[e^{L_{\tau_D}}]\ge \exp({- \sup_{x \in D}\int_0^{\infty}\int_{D}p^D(t,x,y)|\nu|(dy)dt})\ge c_1.
\end{split}
\end{eqnarray}
for some constant $c_1>0$. Combining (\ref{5.7}), (\ref{5.22}) with Theorem 4.6 in \cite{ChungZhao} we obtain (\ref{5.14}).

\vskip 0.3cm

Finally we prove (\ref{5.15}).
In view of (\ref{3.44}), we have
\begin{eqnarray}\label{5.24}
\begin{split}
\inf_{x \in D}E_x[e^{\inf_{s \le 1}L_s}]\ge \exp({- \sup_{x \in D}\int_0^{1}\int_{\Rd}p(t,x,y)|\nu|(dy)dt}) \ge c_2,\\
\end{split}
\end{eqnarray}
for some constant $c_2>0$.
By Cauchy-Schwarz inequality and Proposition \ref{P3.5},
\begin{eqnarray}\label{5.25}
\begin{split}
\sup_{x \in \Rd}E_{x}[\int_0^{1}e^{L_s}d|V|_s]&\le\sup_{x \in \Rd}E_{x}[e^{|L|_1}|V|_1]\\
&\le \sup_{x \in \Rd}(E_{x}[e^{2|L|_1}])^{\frac{1}{2}}\sup_{x \in \Rd}(E_{x}[|V|_1^2])^{\frac{1}{2}} \le c_3<\infty,\\
\end{split}
\end{eqnarray}
for some constant $c_3>0$.
Using (\ref{3.62}), (\ref{5.7}), (\ref{5.14}), (\ref{5.24}) and (\ref{5.25}), we get
\begin{eqnarray} \nonumber
\begin{split}
\sup_{x \in D}E_x[\int_0^{\tau_D}e^{L_s}d|V|_s]& \le \sup_{x \in D} \sum_{n\ge 0}E_x[\int_n^{n+1}e^{L_s}d|V|_s;\tau_D> n]\\
& = \sup_{x \in D} \sum_{n\ge 0}E_x[e^{L_n}E_{X_{n}}[\int_0^{1}e^{L_s}d|V|_s];\tau_D> n]\\
& \le c_3\sup_{x \in D}\sum_{n\ge 0}E_x[e^{L_n};\tau_D> n]\\
& \le c_2^{-1} c_3 \sup_{x \in D}\sum_{n\ge 0}E_x[e^{L_n}E_{X_n}[e^{\inf_{s \le 1}L_s}];\tau_D> n]\\
& = c_2^{-1} c_3 \sup_{x \in D}\sum_{n\ge 0}E_x[e^{\inf_{s\le 1}L_{n+s}};\tau_D> n]\\
& \le c_2^{-1} c_3 \sup_{x \in D}E_x[\int_0^{\tau_D+1}e^{L_s}ds]\\
& \le c_2^{-1} c_3 \sup_{x \in D}(E_x[\int_0^{\tau_D}e^{L_s}ds]+E_x[e^{L_{\tau_D}}E_{X_{\tau_D}}[\int_0^1 e^{L_{s}}ds]])\\
& \le c_2^{-1} c_3 \{\sup_{x \in D}E_x[\int_0^{\tau_D}e^{L_s}ds]+\sup_{x \in D}E_x[e^{L_{\tau_D}}]\sup_{x \in \Rd}E_{x}[e^{|L|_{1}}]\}< \infty.
\end{split}
\end{eqnarray}
\end{proof}

\begin{remark}
When $\nu,\ \rho \in \K_{d,2}$, Lemma \ref{L5.1} is still valid.
\end{remark}

\vskip 0.4cm
Recall that $R_0$ is the diameter of $D$. Set
\begin{eqnarray}\label{5.28}
\begin{split}
u(x):=E_x[e^{L_{\tau_D}}\varphi(X_{\tau_D})+\int_0^{\tau_D}e^{L_s}dV_s].
\end{split}
\end{eqnarray}

Now we are ready to state the main result in this paper:

\vskip 0.4cm

\begin{thm}
Let $\varphi \in C^{1,\a_0}(\partial D)$. Assume there exists a $x_0 \in D$ such that
\begin{eqnarray} \nonumber
\begin{split}
E_{x_0}[e^{L_{\tau_D}}]< \infty.
\end{split}
\end{eqnarray}
Then $u$ defined in (\ref{5.28}) is a weak solution to the Dirichlet boundary value problem (\ref{5.1}).
If for each $1 \le i \le d$, $\mu_i$, $\nu$ and $\p$ belong to $\K_{d, 1-\a_0}$, then $u$ is the unique weak solution to problem (\ref{5.1}) in $C^{1,\a_0}(D)$. Moreover, $u \in C^{1,\a_0}(\overline{D})$ and there exists a $c>0$, depending on $\mu$ and $\nu$ only via the function $\max_{1 \le i \le d}M_{\mu_i}^{1-\a_0}(\cdot)$ and $M_{\nu}^{1-\a_0}(R_0)$, such that
\begin{eqnarray} \label{5.17}
\begin{split}
||u||_{C^{1,\a_0}(\overline{D})}\le c(||\varphi||_{C^{1,\a_0}(\partial D)}+M_{\p}^{1-\a_0}(R_0)+||u||_{C(\overline{D})}).
\end{split}
\end{eqnarray}
\end{thm}

\begin{proof}
By theorem \ref{T4.1}, we know that $u_1:=E_x[\varphi(X_{\tau_D})+V_{\tau_D}]$ is a weak solution to the following problem:
\begin{align} \label{5.18}
\left\{
\begin{aligned}
& \frac{1}{2}\Delta u_1(x)+ \nabla u_1(x) \cdot  \mu(dx) =-\p(dx),&&\forall x \in D, \\
& u_1(x)|_{\partial D}=\varphi(x), && \forall x \in \partial D. \\
\end{aligned}
\right.
\end{align}

Then $u$ defined in (\ref{5.28}) is a weak solution to problem (\ref{5.1}) if and only if $u_2:=u-u_1$ is a weak solution to the following problem:
\begin{align} \label{5.34}
\left\{
\begin{aligned}
&\frac{1}{2}\Delta u_2 + \nabla u_2  \cdot  \mu =-u\nu,&&\forall x \in D, \\
& u_2(x)|_{\partial D}=0, && \forall x \in \partial D. \\
\end{aligned}
\right.
\end{align}

Set $\tilde{u}_2(x):=E_x[\int_0^{\tau_D}u(X_t)dL_t]$. Then it follows from Proposition \ref{P3.2}, Theorem \ref{T4.1} and Lemma \ref{L5.1} that $\tilde{u}_2$ is a weak solution to the problem (\ref{5.34}). To show that $u$ is a weak solution to problem (\ref{5.1}), it is sufficient to prove $\tilde{u}_2=u_2$.

By Theorem 5.13 in \cite{HeYan} and the Fubini's theorem,
\begin{eqnarray} \nonumber
\begin{split}
&\ \ \ \ \tilde{u}_2(x)\\
&=E_x[\int_0^{\tau_D}u(X_t)dL_t]\\
&=E_x[\int_0^{\tau_D}E_{X_t}[e^{L_{\tau_D}}\varphi(X_{\tau_D})+\int_0^{\tau_D}e^{L_s}dV_s]dL_t]\\
&=E_x[\int_0^{\tau_D}E_x[e^{L_{\tau_D}-L_{t\wedge \tau_D}}\varphi(X_{\tau_D})+\int_0^{\tau_D-t\wedge \tau_D}e^{L_{{t\wedge \tau_D}+s}-L_{t\wedge \tau_D}}dV_{{t\wedge \tau_D}+s}|\mathcal{F}_{t\wedge \tau_D} ]dL_t]\\
&=E_x[\int_0^{\tau_D} (e^{L_{\tau_D}-L_{t}}\varphi(X_{\tau_D})+\int_0^{\tau_D-t}e^{L_{t+s}-L_t}dV_{t+s} )dL_t]\\
&=E_x[\varphi(X_{\tau_D})(e^{L_{\tau_D}}-1)]+   E_x[\int_0^{\tau_D}e^{-L_t} \int_t^{\tau_D}e^{L_{s}}dV_s dL_t]\\
&=E_x[\varphi(X_{\tau_D})(e^{L_{\tau_D}}-1)]+   E_x[\int_0^{\tau_D} e^{L_{s}}\int_0^{s}e^{-L_t}dL_tdV_s ]\\
&=E_x[\varphi(X_{\tau_D})(e^{L_{\tau_D}}-1)]+   E_x[\int_0^{\tau_D} (e^{L_{s}}-1)dV_s ]\\
&=E_x[e^{L_{\tau_D}}\varphi(X_{\tau_D})+\int_0^{\tau_D}e^{L_t}dV_t]-E_x[\varphi(X_{\tau_D})+V_{\tau_D}]\\
&=u(x)-u_1(x)=u_2(x).\\
\end{split}
\end{eqnarray}
This proves the existence.
\vskip 0.3cm

Now we assume for each $1 \le i \le d$, $\mu_i,\nu,\p \in \K_{d,1-\a_0}$. Then by Lemma \ref{L4.3}, Lemma \ref{L4.4} and the fact that $u_1,u_2$ are solutions to the problems (\ref{5.18}) and (\ref{5.34}), hence $u=u_1+u_2 \in C^{1,\a_0}(\overline{D})$ and (\ref{5.17}) holds.

\vskip 0.3cm

Next we prove the uniqueness.
Assume $u_1 ,u_2 \in C^{1,\a_0}(D)$ are solutions to problem (\ref{5.1}). Let $v:=u_1-u_2$. Then $v$ satisfies the following equation:
\begin{align} \nonumber
\left\{
\begin{aligned}
&\frac{1}{2}\Delta v + \nabla v  \cdot  \mu=- v  \nu,&&\forall x \in D, \\
& v(x)|_{\partial D}=0, && \forall x \in \partial D. \\
\end{aligned}
\right.
\end{align}

By Proposition \ref{P3.2} and Theorem \ref{T4.2}, we have $v(x)=E_x[\int_0^{\tau_D}v(X_s)dL_s]$. By the strong Markov property of $X$, it is easy to see that
\begin{eqnarray} \nonumber
\begin{split}
v(X_{t \wedge \tau_D})=E_x[\int_0^{\tau_D}v(X_s)dL_s |\mathcal{F}_{t \wedge \tau_D}]-\int_0^{t\wedge \tau_D}v(X_s)dL_s.
\end{split}
\end{eqnarray}
Note that $M_t:=E_x[\int_0^{\tau_D}v(X_s)dL_s |\mathcal{F}_{t \wedge \tau_D}]$ is a martingale. Using the integration by parts formula, we have
\begin{eqnarray}\label{5.30} \nonumber
\begin{split}
v(X_{t \wedge \tau_D})e^{L_{t \wedge \tau_D}}-v(x)&=\int_0^{t \wedge \tau_D}v(X_s)e^{L_s}dL_s+\int_0^{t \wedge \tau_D}e^{L_s}dM_s-\int_0^{t\wedge \tau_D}v(X_s)e^{L_s}dL_s\\
&=\int_0^{t \wedge \tau_D}e^{L_s}dM_s.\\
\end{split}
\end{eqnarray}
Take expectation to get $v(x)=E_x[v(X_{t\wedge \tau_D})e^{L_{t\wedge \tau_D}}]$. Hence, in view of (\ref{5.22}),
\begin{eqnarray} \label{5.23}
\begin{split}
|v(x)|&= |E_x[v(X_{t \wedge \tau_D})e^{L_{t \wedge \tau_D}}]|\\
&\le E_x[|v|(X_{t \wedge \tau_D})e^{L_{t \wedge \tau_D}}]\\
&\le c_1^{-1}E_x[|v|(X_{t \wedge \tau_D})e^{L_{t \wedge \tau_D}}E_{X_{t \wedge \tau_D}}[e^{L_{\tau_D}}]]\\
&\le c_1^{-1}E_x[|v|(X_{t \wedge \tau_D})e^{L_{\tau_D}}].\\
\end{split}
\end{eqnarray}
Since $v \in C(\overline{D})$ and $v$ vanishs on $\partial D$, by (\ref{5.7}) and the dominated convergence theorem, let $t \to \infty$ in (\ref{5.23}) to obtain $v(x)=0$, proving the uniqueness.
\end{proof}

\vskip 0.4cm

\noindent{\bf  Acknowledgement.}\   This work is partly
supported by National Natural Science Foundation of China (No.11671372, No.11431014, No.11401557).
%%%%%%%%%%%%%%%%%%%%%%%%%%%%%%%%%%%%%%%%%%%%%%%%%%%%%%%%%%%%%%%%%%%%%%%%%%%%%%%%%%%%%%%

\end{document}